\documentclass[11pt]{amsart}

\usepackage[margin=1.05in]{geometry}
\usepackage{amsmath,amssymb,amsthm,mathtools}
\usepackage{mathrsfs}
\usepackage{enumitem}
\usepackage{microtype}
\usepackage[colorlinks=true,linkcolor=blue,citecolor=blue,urlcolor=blue]{hyperref}

\numberwithin{equation}{section}

\newtheorem{theorem}{Theorem}[section]
\newtheorem{conjecture}[theorem]{Conjecture}
\newtheorem{proposition}[theorem]{Proposition}
\newtheorem{lemma}[theorem]{Lemma}
\newtheorem{corollary}[theorem]{Corollary}
\newtheorem{question}[theorem]{Question}
\newtheorem*{bryanttheorem}{Bryant's theorem}
\newtheorem*{Lutheorem}{Lu's first pinching theorem}
\newtheorem*{DGLYtheorem}{Ding--Ge--Li--Yang's theorem}
\newtheorem*{GLZtheorem}{Ge--Li--Zhang's theorem}
\theoremstyle{remark}
\newtheorem{remark}[theorem]{Remark}
\newtheorem*{acknow}{Acknowledgments}

\newcommand{\R}{\mathbb R}
\newcommand{\Z}{\mathbb Z}
\newcommand{\Q}{\mathbb Q}
\newcommand{\C}{\mathbb C}
\newcommand{\Sph}{\mathbb S}
\newcommand{\ip}[2]{\left\langle #1,#2\right\rangle}

\newcommand{\Area}{\operatorname{Area}}

\newcommand{\Tr}{\operatorname{Tr}}
\newcommand{\Conv}{\operatorname{Conv}}

\title[Flat minimal tori and Lu's second-gap conjecture]{Flat minimal tori and Lu's second-gap conjecture}
\author[F. G. Li]{Fagui Li}
\address{Frontier Interdisciplinary Domain, Beijing Institute of Technology, Zhuhai, Guangdong 519088, P. R. China}
\email{lifagui@bitzh.edu.cn}

\author[Y. H. Zhao]{Yuhang Zhao$^*$}
\address{School of Mathematics, Nanjing University, Nanjing 210093, P. R. China}
\email{yuhangzhao@smail.nju.edu.cn}

\thanks{F. G. Li is partially supported by NSFC (No. 12271040 and 12501061),
 the Guangdong Provincial Association for Science and Technology Youth Talent Support Program (No. SKXRC2026413) and the Research Start-up Funding of Beijing Institute of Technology (No. 5640011253301)}
 \thanks{$^*$ Corresponding author.}
\begin{document}

\begin{abstract}
Lu conjectured that, for each dimension and codimension, there exists a positive
gap above the first pinching value for the quantity $S+\lambda_2$ on closed
minimal submanifolds of the unit sphere, where $S$ is the squared norm of the
second fundamental form and $\lambda_2$ is the second eigenvalue of Lu's
fundamental matrix. We disprove this second-gap conjecture for minimal surfaces
in every codimension $q\ge3$. More precisely, in every odd codimension $q\ge3$
we construct linearly full closed embedded flat minimal tori with $S\equiv2$ for
which the constant values of $S+\lambda_2$ are dense in $(2,3)$. Thus the first
pinching value $2$ can be approached from above by closed embedded minimal
surfaces, and no uniform second gap exists in codimension at least three.
Together with the known positive results in codimensions one and two, our
examples complete the codimension picture for Lu's second-gap problem for
minimal surfaces: the conjecture holds precisely in codimensions $1$ and $2$,
and fails in every codimension $q\ge3$.
\end{abstract}

\maketitle
\tableofcontents

\section{Introduction}
Let $F:M^n\longrightarrow \Sph^{n+q}$ be a minimal immersion of a closed
$n$-dimensional manifold into the unit sphere. Throughout, ``closed'' means
compact without boundary. We denote by $h$ the second fundamental form of $F$
and put
\[
    S=|h|^2.
\]
For a local orthonormal normal frame $\{\xi_\alpha\}_{\alpha=1}^q$, let
$S_\alpha$ be the corresponding shape operator. Lu's fundamental matrix is the
positive semidefinite symmetric matrix
\[
    \mathcal A=(a_{\alpha\beta})_{1\le\alpha,\beta\le q},
    \qquad
    a_{\alpha\beta}=\langle S_\alpha,S_\beta\rangle
    =\Tr(S_\alpha S_\beta),
\]
where the inner product is the Hilbert--Schmidt inner product on symmetric
endomorphisms of $T_pM$. We order its eigenvalues by
\[
    \lambda_1\ge\lambda_2\ge\cdots\ge\lambda_q\ge0.
\]
When $q=1$, we use the standard convention $\lambda_2:=0$. The trace identity
gives
\[
    S=\Tr\mathcal A=\sum_\alpha\lambda_\alpha.
\]
Thus $S+\lambda_2$ refines the total size $S$ by also recording the
second-largest normal component of $h$ after diagonalizing its normal Gram
matrix.

The quantity $S+\lambda_2$ entered the subject through Lu's work on the normal
scalar curvature conjecture. De Smet, Dillen, Verstraelen, and Vrancken
\cite{DDVV99} proposed the pointwise DDVV inequality
\[
    \rho+\rho^\perp\le \|H\|^2+\kappa
\]
for submanifolds of real space forms $N^{n+m}(\kappa)$, where $\rho$ is the
normalized scalar curvature, $H$ is the mean curvature vector, and $\rho^\perp$
is the normal scalar curvature. Dillen, Fastenakels, and Van der Veken
\cite{DFV07} reformulated this as the matrix inequality
\[
    \sum_{r,s=1}^m\|[B_r,B_s]\|^2
    \le
    \left(\sum_{r=1}^m\|B_r\|^2\right)^2
\]
for real symmetric matrices $B_1,\ldots,B_m$. The conjecture was proved
independently by Ge--Tang \cite{GeTang} and Lu \cite{Lu}; see also
\cite{Ge14,GLLZ20,GLZ21} for related DDVV-type inequalities. Using a refined
commutator inequality, Lu proved the following first-gap theorem.

\begin{Lutheorem}[{\cite{Lu}}]
Let $M^n$ be a closed immersed minimal submanifold of $\Sph^{n+q}$. If
\[
    0\le S+\lambda_2\le n,
\]
then $M$ is totally geodesic, a Clifford product in the hypersurface case, or
the Veronese surface in the exceptional two-dimensional case.
\end{Lutheorem}
This theorem refines the classical pinching theorems of Simons \cite{Simons}
and Chern--do Carmo--Kobayashi \cite{CDK}, as well as later rigidity results of
Yau \cite{Yau1974,Yau1975}, Shen \cite{Shen}, and Li--Li \cite{LiLi}.
Leng--Xu \cite{LengXu} subsequently obtained a generalized Lu rigidity theorem
for submanifolds with parallel mean curvature.

At the end of Section~5 of \cite{Lu}, Lu proposed the following second-gap
conjecture.
\begin{conjecture}[Lu's second-gap conjecture]
\label{conj:Lu}
Let $M^n$ be a closed immersed minimal submanifold of $\Sph^{n+q}$. Assume that
$S+\lambda_2$ is constant and
\[
    S+\lambda_2>n.
\]
Then there is a constant $\varepsilon(n,q)>0$ such that
\[
    S+\lambda_2>n+\varepsilon(n,q).
\]
\end{conjecture}

In codimension one, Lu's matrix has only one normal direction. With the
standard convention $\lambda_2=0$, Lu's quantity becomes $S$, and
Conjecture~\ref{conj:Lu} becomes the second-gap form of Chern's conjecture for
minimal hypersurfaces. In higher codimension, however, Chern's problem and
Lu's problem are different: Chern's question concerns possible constant values
of $S$ itself, whereas Lu's conjecture concerns the refined quantity
$S+\lambda_2$.

For reference we recall the higher-codimensional version of the Chern
discreteness problem.
\begin{conjecture}[Chern's conjecture in higher codimension]
\label{conj:chern-higher}
Let $M^n\subset\Sph^{n+q}$ be a closed minimal submanifold with constant scalar
curvature, equivalently with constant $S=|h|^2$.
\begin{enumerate}[label=\textup{(\roman*)},leftmargin=2em]
\item For fixed $n$ and $q$, the set of possible constant values of $S$ is
   discrete.
\item There exists a constant
\[
    S_0>\frac{n}{2-1/q}
\]
such that
\[
    \frac{n}{2-1/q}<S\le S_0
    \quad\Longrightarrow\quad
    S=S_0.
\]
\end{enumerate}
\end{conjecture}
The problem goes back to Chern \cite{CDK}; see also the surveys of Scherfner--Weiss--Yau \cite{yausurvey} and
Xu--Xu \cite{XuXuSurvey}.
Simons' first-gap theorem gives the critical value
$n/(2-1/q)$ in codimension $q$. Subsequent pinching results of Shen
\cite{Shen}, Li--Li \cite{LiLi}, and Chen--Xu \cite{ChenXu} motivate the
second-gap question above this first pinching level. In the hypersurface case,
Chern's problem and its second-gap forms have been studied by Peng--Terng
\cite{pengterng83,pengterngseminar}, Chang \cite{chang}, Yang--Cheng
\cite{YangCheng98}, Ding--Xin \cite{dingxin}, Xu--Xu \cite{xuxu13,xuxu17}, and
Deng--Gu--Wei \cite{deng}, among others.
Recent progress on Chern-type and isoparametric
rigidity problems includes Tang--Wei--Yan \cite{TangWeiYan}, Tang--Yan
\cite{TangYan}, He--Xu--Zhao \cite{HeXuZhao}, Ge--Liu--Luo--Yan
\cite{GeLiuLuoYan}, Ge--Tan--Yan--Zhang \cite{GeTanYanZhang}, and Tao
\cite{Tao26}.

The counterexamples in this paper do not address Conjecture~\ref{conj:chern-higher}.
Indeed, along our whole family one has
\[
    S\equiv2.
\]
Only $\lambda_2$ varies. Thus the examples disprove a uniform second gap for
Lu's refined quantity without producing any new constant value of $S$.
%For a minimal surface in the unit sphere, the Gauss equation is
%\[
%    K=1-\frac{S}{2},
%\]
%so intrinsic flatness is exactly the condition $S\equiv2$. 
The examples
below are closed embedded flat minimal tori. They should not be confused with
surfaces having flat normal bundle, as considered in other pinching theorems:
in our examples the normal scalar curvature is positive away from the limiting
degeneration.

For comparison, we recall the following positive results of Ding--Ge--Li--Yang
\cite{DGLY}; see also Ding--Ge--Li \cite{dinggeli25,DingGeLiSimon2026allgap} for related pinching
rigidity results for minimal surfaces in spheres.
\begin{DGLYtheorem}[\cite{DGLY}]
Let $M^2$ be a closed minimal surface immersed in $\Sph^{2+q}$. The following
assertions hold.
\begin{enumerate}[label=\textup{(\roman*)},leftmargin=2em]
\item If $M^2$ is a linearly full minimal $2$-sphere and $S+\lambda_2$ is
constant, then the Gaussian curvature $K$ is constant and $M$ is one of
Calabi's minimal $2$-spheres \cite{Calabi}.
\item If $M^2$ is a linearly full minimal $2$-sphere and $S+\lambda_2>2$, then
\[
    \max_{p\in M}(S+\lambda_2)(p)\ge\frac52.
\]
Equality holds if and only if $q=4$ and $M$ is Calabi's minimal $2$-sphere with
$K=1/6$.
\item If $S+\lambda_2>2$ and
$
    \rho^\perp
    \le
    \frac12\sqrt{(S+\lambda_2-2)\gamma S}
$
 for some $0\le\gamma\le2/3$, then
\[
    \max_{p\in M}(S+\lambda_2)(p)
    \ge \max_{p\in M}S(p)
    >\frac{40+12\gamma}{18+9\gamma}
    \ge2.
\]
In particular, if the normal bundle is flat, equivalently $\rho^\perp=0$, and
$S+\lambda_2>2$, then
\[
    \max_{p\in M}(S+\lambda_2)(p)>\frac{20}{9}.
\]
\end{enumerate}
\end{DGLYtheorem}
More recently, Ge--Li--Zhang \cite{GeLiZhang26} settled the  
codimension-two case.
\begin{GLZtheorem}[{\cite{GeLiZhang26}}]
Let $M^2$ be a closed minimal surface immersed in $\mathbb{S}^{4}$ and suppose that
\[
        S+\lambda_2\equiv c
\]
for a constant $c$.  Then $c\in\{0,2\}$.  More precisely, exactly one of the following occurs:
\begin{enumerate}
\item[(i)] $c=0$ and $M$ is a totally geodesic $2$-sphere;
\item[(ii)] $c=2$ and $M$ is a Clifford torus in a totally geodesic
$\Sph^3\subset\Sph^4$;
\item[(iii)] $c=2$ and $M$ is the Veronese surface in $\Sph^4$.
\end{enumerate}
\end{GLZtheorem}
Thus Lu's second-gap conjecture holds for closed minimal surfaces in codimension
two. Together with the hypersurface results of Peng--Terng
\cite{pengterng83,pengterngseminar} and the counterexamples constructed in the
present paper, this completes the codimension picture of Lu's second-gap
conjecture for minimal surfaces: it holds precisely in codimensions $q=1,2$ and
fails in every codimension $q\ge3$.

The positive results above show that Lu's conjecture survives under genus-zero
or normal-curvature pinching hypotheses in arbitrary codimension, and without
such auxiliary hypotheses in codimension two. By contrast, the present paper
shows that, once such additional hypotheses are removed, Conjecture~\ref{conj:Lu}
fails already in the class of closed embedded flat tori in codimension at least
three. It is important to stress, however, that our counterexamples are
two-dimensional flat tori. They do not rule out the possibility that Lu's
second-gap conclusion remains valid in other geometric classes, for instance
for higher-dimensional minimal submanifolds or for simply connected minimal
submanifolds. In these directions, the conjecture remains an interesting open
rigidity problem.

The construction is explicit. With the sign convention
$\Delta=\operatorname{div}\nabla$, Takahashi's theorem \cite{Takahashi} says
that an isometric immersion $F:M^n\to\Sph^{n+q}$ is minimal if and only if
\[
    \Delta F+nF=0.
\]
On a flat torus, such maps are finite sums of sine--cosine frequency pairs of a
common length. Bryant's flat-torus criterion, recalled in
Section~\ref{sec:construction-computations}, explains the corresponding
convex-hull condition. Recent work on minimal immersions of flat tori includes
\cite{LWX,LWXfirst}.

Our first main theorem is a density statement.
\begin{theorem}
\label{thm:main}
Let $q\ge3$ be odd. % and write $q=2m-3$. 
Among closed linearly full
embedded flat minimal tori in $\Sph^{2+q}$, the set of constant values of
$S+\lambda_2$ is dense in $(2,3)$. Consequently, the analogue of Chern's
discreteness statement fails for Lu's refined quantity $S+\lambda_2$ in every
odd codimension $q\ge3$.
\end{theorem}

The next theorem records an explicit sequence approaching the first-gap value.
\begin{theorem}[Odd-codimensional flat-torus sequence]
\label{thm:sequence}
Let $m\ge3$ and put $q=2m-3$. There is a sequence of closed, linearly full,
embedded flat minimal tori
\[
    \widehat F_j:\widehat T_j^2\longrightarrow \Sph^{2m-1}=\Sph^{2+q}
\]
such that, for
\[
    R_j=\sum_\alpha w_{\alpha,j}e^{4i\theta_{\alpha,j}},
\]
where $w_{\alpha,j}$ are the weights and $\theta_{\alpha,j}$ are the frequency
angles,
\[
    S\equiv2,
    \quad
    \lambda_1\equiv 1+|R_j|,
    \quad
    \lambda_2\equiv 1-|R_j|,
    \quad
    \rho_0^\perp
    =\sum_{\alpha,\beta}\|[S_\alpha,S_\beta]\|^2
    \equiv4(1-|R_j|^2).
\]
Moreover $0<|R_j|<1$ and $|R_j|\to1$. Hence
\[
    S+\lambda_2\equiv3-|R_j|>2,
    \qquad
    S+\lambda_2\longrightarrow2.
\]
\end{theorem}

\begin{corollary}
\label{cor:all-q-nonfull}
For every odd $q\ge3$, there is no constant $\varepsilon(q)>0$ such that every
linearly full closed embedded \emph{flat minimal torus} in $\Sph^{2+q}$ with
constant $S+\lambda_2>2$ must satisfy
\[
    S+\lambda_2>2+\varepsilon(q).
\]
If the fullness assumption is dropped, the same failure holds among closed
embedded flat minimal tori in every codimension $q\ge3$.
\end{corollary}

Finally, we prove an exact identity which explains the role of area.
\begin{theorem}
\label{thm:genus-area}
Let $M^2\subset \Sph^{2+q}$ be a closed connected orientable immersed minimal surface of
genus $g$; no flatness or torus assumption is imposed in this statement. In
codimension one we use the convention $\lambda_2=0$. Assume that
\[
    S+\lambda_2\equiv c
\]
is constant. Then
\begin{equation}\label{eq:gb-identity-intro}
    c=2+
    \frac{8\pi(g-1)}{\Area(M)}
    +\frac{1}{\Area(M)}\int_M\lambda_2\,d\mu.
\end{equation}
In particular, let $\lambda_1^\Delta(M)>0$ denote the first positive eigenvalue of
the nonnegative operator $-\Delta$ for the induced metric, equivalently
$\Delta u+\lambda_1^\Delta(M)u=0$ for a first eigenfunction. If $g>2$, then
\[
    c>2+\frac{g-1}{\left\lfloor\frac{g+3}{2}\right\rfloor}
    \lambda_1^\Delta(M).
\]
If $g=2$, then
$
    c\ge 2+\frac12\lambda_1^\Delta(M).
$
\end{theorem}

The last theorem is area-dependent. For $g\ge2$, the topological term in
\eqref{eq:gb-identity-intro} degenerates as $\Area(M)\to\infty$, and the
eigenvalue form degenerates when $\lambda_1^\Delta(M)\to0$. Therefore the
identity does not by itself provide a uniform second gap in higher genus.
Whether an area-escape phenomenon analogous to that of the flat-torus examples
can occur in genus at least two remains open. When $g=1$, the topological term
vanishes. Along each fixed-$\ell$ sequence in the explicit finite-covering
family constructed below, the gap $S+\lambda_2-2$ is of reciprocal-area order
on the chosen covering domains, while the area of the embedded quotient also
depends on the covering degree.

\section{Preliminaries}
\label{sec:preliminaries}

Let
$
    F:M^2\longrightarrow \Sph^{2+q}
$
be an immersed surface. We use tangent indices $i,j,k\in\{1,2\}$ and normal
indices $\alpha,\beta\in\{1,\ldots,q\}$. Given a local orthonormal tangent frame
$e_1,e_2$ and a local orthonormal normal frame $\xi_1,\ldots,\xi_q$, write
\[
    h(e_i,e_j)=\sum_{\alpha=1}^q h_{ij}^\alpha\xi_\alpha,
    \qquad h_{ij}^\alpha=h_{ji}^\alpha.
\]
Equivalently, if $S_\alpha$ is the shape operator in the normal direction
$\xi_\alpha$, then
\[
    S_\alpha=(h_{ij}^\alpha)_{1\le i,j\le2}.
\]
The second fundamental form is $h=\sum_{i,j,\alpha}h_{ij}^\alpha\omega_i
\omega_j\xi_\alpha$, and the mean curvature vector is
\[
    H=\frac12\sum_{\alpha=1}^q(h_{11}^\alpha+h_{22}^\alpha)\xi_\alpha.
\]
The immersion is minimal precisely when $H\equiv0$.

We shall also use the following matrix notation for vectors in $\R^2$. The
symbol $I_2$ denotes the $2\times2$ identity matrix. If $u=(u_1,u_2)\in\R^2$, then
$u\otimes u$ denotes the rank-one symmetric matrix
\[
    u\otimes u=
    \begin{pmatrix}
    u_1^2 & u_1u_2\\
    u_1u_2 & u_2^2
    \end{pmatrix}.
\]
Equivalently, $u\otimes u$ is the quadratic form $(u\cdot dx)^2$ on $\R^2$. Thus
an identity such as
\[
    \sum_\alpha w_\alpha u_\alpha\otimes u_\alpha=\frac12 I_2
\]
means the three scalar equations obtained by comparing the $(1,1)$, $(2,2)$, and
$(1,2)$ entries.
\subsection{Lu's fundamental matrix for a minimal surface}

Let $M^2\subset \Sph^{2+q}$ be a minimal surface in the unit sphere. For a normal vector $\xi$,
the shape operator $S_\xi$ is defined by
\[
    \langle S_\xi X,Y\rangle=\langle h(X,Y),\xi\rangle
\]
for tangent vector fields $X, Y$ on $M$.
Given a local orthonormal normal frame $\{\xi_\alpha\}_{\alpha=1}^q$, Lu's
fundamental matrix is
\[
    \mathcal A=(a_{\alpha\beta})_{1\le\alpha,\beta\le q},
    \qquad
    a_{\alpha\beta}=\langle S_\alpha,S_\beta\rangle
    =\Tr(S_\alpha S_\beta).
\]
If the normal frame is changed by an orthogonal matrix, then $\mathcal A$ is
orthogonally conjugated. Hence its eigenvalues
\[
    \lambda_1\ge\lambda_2\ge\cdots\ge\lambda_q\ge0
\]
are globally well-defined.

At a fixed point, choose an orthonormal tangent frame $e_1,e_2$. Since the
surface is minimal, there are normal vectors $a,b\in N_pM$ such that
\[
    h_{11}=a,
    \qquad
    h_{12}=b,
    \qquad
    h_{22}=-a.
\]
After identifying $N_pM$ with $\R^q$ and writing
$a=(a_1,\ldots,a_q)$, $b=(b_1,\ldots,b_q)$, the shape operator in the normal
direction $\xi_\alpha$ is
\[
    S_\alpha=
    \begin{pmatrix}
    a_\alpha & b_\alpha\\
    b_\alpha & -a_\alpha
    \end{pmatrix}.
\]
Consequently
\[
    S=|h|^2=2|a|^2+2|b|^2
\]
and
\begin{equation}
\label{eq:fundamental-matrix}
    \mathcal A=2aa^T+2bb^T.
\end{equation}
In particular $\operatorname{rank}\mathcal A\le2$. The two possible nonzero
eigenvalues, still denoted by $\lambda_1\ge\lambda_2\ge0$, satisfy
\begin{equation}\label{eq:surface-trace-lu}
    S=\Tr \mathcal A=\lambda_1+\lambda_2.
\end{equation}
This identity is used repeatedly below.

We shall distinguish two conventions for normal scalar curvature. The
unnormalized quantity is
\[
    \rho_0^\perp:=\sum_{\alpha,\beta=1}^q\|[S_\alpha,S_\beta]\|^2.
\]
The normalized quantity is
\begin{equation}
\label{eq:rho-normalized}
    \rho^\perp=\frac{1}{n(n-1)}\sqrt{\rho_0^\perp}
\end{equation}
for an $n$-dimensional submanifold. In this paper most computations are for
surfaces, but keeping both notations prevents convention conflicts.

For a minimal surface, a direct computation gives
\[
    [S_\alpha,S_\beta]
    =2(a_\alpha b_\beta-b_\alpha a_\beta)
    \begin{pmatrix}0&1\\ -1&0\end{pmatrix}.
\]
Since
\[
    \sum_{\alpha,\beta=1}^q
    (a_\alpha b_\beta-b_\alpha a_\beta)^2
    =2\bigl(|a|^2|b|^2-\langle a,b\rangle^2\bigr),
\]
we obtain
\[
    \rho_0^\perp
    =16\bigl(|a|^2|b|^2-\langle a,b\rangle^2\bigr).
\]
By the preceding rank-two formula, equivalently by Lemma~3.2 in \cite{DGLY}, the two possible nonzero eigenvalues of $\mathcal A$ are twice the eigenvalues of the
Gram matrix
\[
    \begin{pmatrix}
    |a|^2&\langle a,b\rangle\\
    \langle a,b\rangle&|b|^2
    \end{pmatrix}.
\]
Thus
\[
    \rho_0^\perp=4\lambda_1\lambda_2
    =S^2-(\lambda_1-\lambda_2)^2.
\]
These elementary identities are the only local formulas for Lu's matrix needed
in the construction.

\section{Flat tori and Lu's fundamental matrix}
\label{sec:construction-computations}

This section gives the construction of the flat tori and computes Lu's fundamental matrix. We keep all metric normalizations explicit.

We first recall Bryant's flat-torus criterion.
A rank-two lattice in $\R^2$ means a discrete subgroup of the form
\[
    A=\Z\omega_1+\Z\omega_2
\]
for two linearly independent vectors $\omega_1,\omega_2\in\R^2$. If
$E=\{z_1,\ldots,z_N\}\subset\C\simeq\R^2$ is a finite set, its convex hull is
\[
    \Conv(E)
    =\left\{\sum_{j=1}^N t_jz_j: t_j\ge0,\ \sum_{j=1}^N t_j=1\right\}.
\]
Equivalently, it is the smallest convex subset of $\C$ containing $E$.

\begin{bryanttheorem}[\cite{Bryant}]
Let $A\subset\R^2\simeq\C$ be a rank-two lattice and let
\[
    A^*=\{(\xi,\eta)\in\R^2: \xi x+\eta y\in2\pi\Z
    \text{ for every }(x,y)\in A\}
\]
be its $2\pi$-dual lattice. For $r>0$, let $H(A,r)$ be the convex hull in
$\C$ of
\[
    C(A,r)=\{(\xi+i\eta)^2:(\xi,\eta)\in A^*,\ \xi^2+\eta^2=r^2\}.
\]
Then $\R^2/A$ admits a minimal immersion into some unit sphere with induced
metric
\[
    \langle df,df\rangle=\frac{r^2}{2}(dx^2+dy^2)
\]
if and only if $0\in H(A,r)$. In particular, such an immersion exists for some
$r>0$ if and only if $A^*$ has a basis $\alpha,\beta\in\C$ such that
\[
    \operatorname{Re}(\alpha/\beta)\in\Q,
    \qquad
    |\alpha|^2/|\beta|^2\in\Q.
\]
\end{bryanttheorem}

\subsection{Odd-codimensional homogeneous flat tori}
\label{sec:flat-construction}

We first give the construction in the form used for all odd codimensions. Let
\[
    q=2m-3,\qquad m\ge3.
\]
Thus the target sphere is
\[
    \Sph^{2+q}=\Sph^{2m-1}\subset\R^{2m}.
\]
The integer $m$ is the number of sine--cosine frequency pairs. Start from the flat two-torus
\begin{equation}
\label{eq:flat-domain-metric}
    T_L^2=(\R/2\pi\Z)^2,
    \qquad
    g_L=\frac{L^2}{2}(dx_1^2+dx_2^2),
\end{equation}
where $L>0$. Let
\[
    v_1,\ldots,v_m\in\Z^2
\]
be nonzero integer vectors, pairwise distinct up to sign, and assume that they have a common Euclidean length
\begin{equation}
\label{eq:common-frequency-length}
    |v_1|=\cdots=|v_m|=L.
\end{equation}
Write
\[
    u_\alpha=\frac{v_\alpha}{L}
    = (\cos\theta_\alpha,\sin\theta_\alpha),
    \qquad \alpha=1,\ldots,m.
\]
Choose positive weights $w_\alpha>0$ satisfying the isometry condition
\begin{equation}
\label{eq:general-isometry}
    \sum_{\alpha=1}^m w_\alpha=1,
    \qquad
    \sum_{\alpha=1}^m w_\alpha u_\alpha\otimes u_\alpha=\frac12 I_2.
\end{equation}
Here $I_2$ is the $2\times2$ identity matrix and
$u_\alpha\otimes u_\alpha$ is the rank-one matrix defined in the preliminaries.
The second identity is precisely the condition that the map constructed below
induce the Euclidean metric: it says that the weighted average of the matrices
$u_\alpha\otimes u_\alpha$ is $\frac12 I_2$.
This condition is also Bryant's convex-hull condition in the
present homogeneous setting. Indeed, write
\[
    u_\alpha=(u_{\alpha1},u_{\alpha2})
    = (\cos\theta_\alpha,\sin\theta_\alpha),
    \qquad
    z_\alpha=u_{\alpha1}+iu_{\alpha2}=e^{i\theta_\alpha}.
\]
Since \(\sum_\alpha w_\alpha=1\), comparison of real and imaginary parts gives
\[
\begin{aligned}
    \sum_\alpha w_\alpha u_\alpha\otimes u_\alpha=\frac12 I_2
    \quad&\Longleftrightarrow\quad
    \sum_\alpha w_\alpha(u_{\alpha1}^2-u_{\alpha2}^2)=0,
    \quad
    \sum_\alpha w_\alpha u_{\alpha1}u_{\alpha2}=0  \\
    \quad&\Longleftrightarrow\quad
    \sum_\alpha w_\alpha z_\alpha^2=0.
\end{aligned}
\]
To connect this with Bryant's formulation for integer frequencies, observe that
\(v_\alpha=Lu_\alpha\). Hence
\[
    (v_{\alpha1}+iv_{\alpha2})^2=L^2z_\alpha^2.
\]
Multiplication by the common positive factor \(L^2\) does not affect whether the
origin is a convex combination. Therefore the moment condition for the weights
\(w_\alpha\) is equivalent to
\[
    \sum_\alpha w_\alpha (v_{\alpha1}+iv_{\alpha2})^2=0.
\]
Because the weights are positive and sum to one, this says precisely that the
origin is a convex combination, with coefficients \(w_\alpha\), of the points
\((v_{\alpha1}+iv_{\alpha2})^2\). In particular,
\[
    0\in\Conv\{(v_{\alpha1}+iv_{\alpha2})^2:1\le\alpha\le m\}.
\]
We shall nevertheless give the direct calculation below, since its
normalizations are used later in the computation of Lu's fundamental matrix.
Define
\begin{equation}
\label{eq:general-eigenmap}
\begin{aligned}
    F(x)=(&\sqrt{w_1}\cos\ip{v_1}{x},\sqrt{w_1}\sin\ip{v_1}{x},\ldots, \\
    &\sqrt{w_m}\cos\ip{v_m}{x},\sqrt{w_m}\sin\ip{v_m}{x}).
\end{aligned}
\end{equation}
The following lemma records the standard homogeneous flat-torus eigenfunction construction with the normalizations used later.

\begin{lemma}
\label{lem:general-eigenmap}
Under the assumptions above, the map
\[
    F:(T_L^2,g_L)\longrightarrow \Sph^{2m-1}
\]
is a linearly full isometric minimal immersion. Its domain is intrinsically flat, and
\begin{equation}
\label{eq:general-area}
    \Area(T_L^2,g_L)=2\pi^2L^2.
\end{equation}
Moreover
\[
    S\equiv2.
\]
\end{lemma}

\begin{proof}
We first check that the formula \eqref{eq:general-eigenmap} defines a map on the torus
$T_L^2=(\R/2\pi\Z)^2$. If $\gamma\in2\pi\Z^2$, then, because each $v_\alpha$ has integer components,
\[
    \langle v_\alpha,x+\gamma\rangle-\langle v_\alpha,x\rangle
    =\langle v_\alpha,\gamma\rangle\in2\pi\Z.
\]
Thus the two functions $\cos\langle v_\alpha,x\rangle$ and
$\sin\langle v_\alpha,x\rangle$ have the same values at $x$ and
$x+\gamma$. This proves that every coordinate function in
\eqref{eq:general-eigenmap} is well-defined on $T_L^2$, and hence so is
$F$.

The map takes values in the unit sphere. Indeed, using the first identity in
\eqref{eq:general-isometry}, we obtain
\[
\begin{aligned}
    |F(x)|^2
    &=\sum_{\alpha=1}^m w_\alpha
    \bigl(\cos^2\langle v_\alpha,x\rangle+\sin^2\langle v_\alpha,x\rangle\bigr) \\
    &=\sum_{\alpha=1}^m w_\alpha
    =1.
\end{aligned}
\]
Therefore $F(x)\in\Sph^{2m-1}$ for every $x\in T_L^2$.

We now compute the induced metric. For a fixed index $\alpha$, let
\[
    \phi_\alpha(x)=\langle v_\alpha,x\rangle,
\]
and consider the pair of coordinate functions
\[
    \sqrt{w_\alpha}(\cos\phi_\alpha,\sin\phi_\alpha)
\]
in the $\alpha$-th coordinate two-plane of $\R^{2m}$. For $i=1,2$, the partial derivative of this pair with respect to $x_i$ is
\[
    \sqrt{w_\alpha}\,v_{\alpha i}(-\sin\phi_\alpha,\cos\phi_\alpha).
\]
Hence the contribution of this two-plane to the inner product of the coordinate
vectors $\partial_iF$ and $\partial_jF$ is
\[
    w_\alpha v_{\alpha i}v_{\alpha j}
    \bigl(\sin^2\phi_\alpha+\cos^2\phi_\alpha\bigr)
    =w_\alpha v_{\alpha i}v_{\alpha j}.
\]
Different coordinate two-planes in $\R^{2m}$ are mutually orthogonal. Therefore
\begin{equation}
\label{eq:general-induced-metric-computation}
    \langle \partial_iF,\partial_jF\rangle
    =\sum_{\alpha=1}^m w_\alpha v_{\alpha i}v_{\alpha j}.
\end{equation}
Since $v_\alpha=Lu_\alpha$, the second identity in \eqref{eq:general-isometry} gives
\[
    \sum_{\alpha=1}^m w_\alpha v_{\alpha i}v_{\alpha j}
    =L^2\sum_{\alpha=1}^m w_\alpha u_{\alpha i}u_{\alpha j}
    =\frac{L^2}{2}\delta_{ij}.
\]
Comparing with the definition of $g_L$ in \eqref{eq:flat-domain-metric}, we get
\[
    F^*g_{\Sph^{2m-1}}=g_L.
\]
Thus $F$ is an isometric immersion. In particular, since $g_L$ is positive
definite, the differential of $F$ has rank two everywhere.

Next we prove minimality. From \eqref{eq:flat-domain-metric}, the
geometric Laplacian of the flat metric $g_L$ is
\begin{equation}
\label{eq:general-flat-laplacian}
    \Delta_{g_L}
    =\frac{2}{L^2}\left(\frac{\partial^2}{\partial x_1^2}
    +\frac{\partial^2}{\partial x_2^2}\right).
\end{equation}
For each $\alpha$, the ordinary Euclidean second derivative gives
\[
    \left(\frac{\partial^2}{\partial x_1^2}
    +\frac{\partial^2}{\partial x_2^2}\right)
    \cos\langle v_\alpha,x\rangle
    =-|v_\alpha|^2\cos\langle v_\alpha,x\rangle,
\]
and the same identity holds with the cosine replaced by the sine. Using the
common-length assumption \eqref{eq:common-frequency-length} in
\eqref{eq:general-flat-laplacian}, we obtain
\[
    \Delta_{g_L}\cos\langle v_\alpha,x\rangle
    =-2\cos\langle v_\alpha,x\rangle,
    \qquad
    \Delta_{g_L}\sin\langle v_\alpha,x\rangle
    =-2\sin\langle v_\alpha,x\rangle.
\]
Applying this to all coordinate functions in \eqref{eq:general-eigenmap}
gives
\begin{equation}
\label{eq:general-takahashi-equation}
    \Delta_{g_L}F+2F=0.
\end{equation}
Since the domain has dimension two, Takahashi's theorem \cite{Takahashi},
together with \eqref{eq:general-takahashi-equation}, implies that $F$ is
minimal in the unit sphere.

We next prove that the immersion is linearly full. Suppose that a constant
vector $A=(A_1,\ldots,A_{2m})\in\R^{2m}$ satisfies
$\langle A,F(x)\rangle\equiv0$ on $T_L^2$. By the formula
\eqref{eq:general-eigenmap}, this means that
\begin{equation}
\label{eq:linear-full-identity}
    \sum_{\alpha=1}^m \sqrt{w_\alpha}
    \bigl(A_{2\alpha-1}\cos\langle v_\alpha,x\rangle
    +A_{2\alpha}\sin\langle v_\alpha,x\rangle\bigr)\equiv0.
\end{equation}
The character orthogonality relation on $(\R/2\pi\Z)^2$ is
\[
    \int_{[0,2\pi]^2}e^{i\langle n,x\rangle}\,dx
    =\begin{cases}(2\pi)^2,&n=0,\\0,&n\in\Z^2\setminus\{0\}.
    \end{cases}
\]
Because the vectors $v_\alpha$ are nonzero and pairwise distinct up to sign,
the characters with frequencies $\pm v_\alpha$ are pairwise distinct. Fix an
index $\beta$. Multiplying \eqref{eq:linear-full-identity} by
$\cos\langle v_\beta,x\rangle$ and integrating over $[0,2\pi]^2$ therefore
leaves only the cosine term with $\alpha=\beta$. Since $w_\beta>0$, we obtain
$A_{2\beta-1}=0$. Multiplication by
$\sin\langle v_\beta,x\rangle$ gives $A_{2\beta}=0$ in the same way. Because
$\beta$ was arbitrary, all components of $A$ vanish.
Thus no nonzero constant vector is orthogonal to the image of $F$, and the image
is not contained in any proper linear subspace of $\R^{2m}$. Therefore $F$ is
linearly full.

It remains to compute the area and the value of $S$. The metric in
\eqref{eq:flat-domain-metric} has constant area element
\[
    d\mu_{g_L}=\sqrt{\det\left(\frac{L^2}{2}I_2\right)}\,dx_1dx_2
    =\frac{L^2}{2}\,dx_1dx_2.
\]
Integrating over the coordinate square $[0,2\pi]^2$ gives
\[
    \Area(T_L^2,g_L)
    =\int_0^{2\pi}\int_0^{2\pi}\frac{L^2}{2}\,dx_1dx_2
    =2\pi^2L^2,
\]
which is \eqref{eq:general-area}.

Finally, the metric $g_L$ has constant coefficients, so the intrinsic Gaussian
curvature of $(T_L^2,g_L)$ is identically zero. For a minimal surface in the
unit sphere, the Gauss equation gives
\[
    K=1-\frac{S}{2}.
\]
Since $K\equiv0$, we obtain $S\equiv2$. This completes the proof.
\end{proof}

The torus used in Lemma~\ref{lem:general-eigenmap} is chosen only for convenience: it makes all
frequency functions single-valued and keeps the computation simple. It may,
however, be a finite cover of the natural torus determined by the frequencies.
The next observation passes to that natural quotient, defined by the full period
lattice, and records that the induced map is embedded.
\begin{proposition}[Embedded quotient]
\label{prop:embedded-quotient}
Let $k_1,\ldots,k_m\in\R^2$ be such that their integer span
\[
    \mathcal L=\langle k_1,\ldots,k_m\rangle_{\Z}\subset\R^2
\]
is a rank-two lattice of frequencies, and let
\[
    \Lambda_{\mathcal L}
    =\{x\in\R^2:\langle k,x\rangle\in2\pi\Z
    \text{ for every }k\in\mathcal L\}
\]
be its period lattice. Assume that
\[
    w_\alpha>0\quad(\alpha=1,\ldots,m),
    \qquad
    \sum_{\alpha=1}^m w_\alpha=1.
\]
Then the map
\[
    F:\R^2/\Lambda_{\mathcal L}\longrightarrow \Sph^{2m-1},
    \qquad
    F(x)=\bigl(\sqrt{w_\alpha}\cos\langle k_\alpha,x\rangle,
    \sqrt{w_\alpha}\sin\langle k_\alpha,x\rangle\bigr)_{\alpha=1}^m
\]
is injective. Consequently, whenever this map is an immersion, it is an
embedding. More generally, if the same formula is considered on
$\R^2/\Gamma$, where $\Gamma\subset\Lambda_{\mathcal L}$ has finite index,
then it is the composition of the finite covering
$\R^2/\Gamma\to\R^2/\Lambda_{\mathcal L}$ with this embedded map.
\end{proposition}

\begin{proof}
Choose a basis $\ell_1,\ell_2$ of the rank-two lattice $\mathcal L$ and let
$B$ be the matrix with columns $\ell_1,\ell_2$. Thus $\mathcal L=B\Z^2$.
Because $\ell_1$ and $\ell_2$ generate $\mathcal L$, the defining condition for
$\Lambda_{\mathcal L}$ is equivalent to imposing it only on these two basis
vectors:
\[
    x\in\Lambda_{\mathcal L}
    \quad\Longleftrightarrow\quad
    \langle \ell_1,x\rangle,\langle \ell_2,x\rangle\in 2\pi\Z.
\]
Since
\[
    B^Tx=
    \begin{pmatrix}
        \langle \ell_1,x\rangle\\
        \langle \ell_2,x\rangle
    \end{pmatrix},
\]
we have
\[
    x\in\Lambda_{\mathcal L}
    \quad\Longleftrightarrow\quad
    B^Tx\in 2\pi\Z^2
    \quad\Longleftrightarrow\quad
    x\in 2\pi (B^T)^{-1}\Z^2.
\]
Here
\[
    B^{-T}:=(B^T)^{-1}=(B^{-1})^T.
\]
Thus
\[
    \Lambda_{\mathcal L}=2\pi B^{-T}\Z^2
    =\left\{2\pi B^{-T}
    \begin{pmatrix}n_1\\ n_2\end{pmatrix}:n_1,n_2\in\Z\right\}.
\]
Equivalently, the two columns of $2\pi B^{-T}$ form the period basis dual to
$\ell_1,\ell_2$ in the $2\pi$ normalization. In particular,
$\Lambda_{\mathcal L}$ is a rank-two lattice and $\R^2/\Lambda_{\mathcal L}$ is a
compact torus.

For $k\in\mathcal L$ set
\[
    \chi_k(x)=e^{i\langle k,x\rangle}.
\]
Then $\chi_k(x+\gamma)=\chi_k(x)$ for every $k\in\mathcal L$ and every
$\gamma\in\Lambda_{\mathcal L}$, by the definition of
$\Lambda_{\mathcal L}$. In particular the characters
$\chi_{k_1},\ldots,\chi_{k_m}$, and hence the coordinate functions of
$F$, are well-defined on $\R^2/\Lambda_{\mathcal L}$. Moreover,
\[
    |F(x)|^2=\sum_{\alpha=1}^m w_\alpha=1,
\]
so the map indeed takes values in the unit sphere.

It remains to prove injectivity. Consider the character map
\[
    \Phi:\R^2\longrightarrow (\mathbb S^1)^m,
    \qquad
    \Phi(x)=(\chi_{k_1}(x),\ldots,\chi_{k_m}(x)).
\]
The map $F$ is obtained from $\Phi$ by replacing each complex number
$\chi_{k_\alpha}(x)$ by its real and imaginary parts and multiplying the
corresponding two coordinates by the positive number $\sqrt{w_\alpha}$.
Since all $w_\alpha$ are positive,
\[
    F(x)=F(y)
    \quad\Longleftrightarrow\quad
    \Phi(x)=\Phi(y).
\]
Thus $F(x)=F(y)$ implies
\[
    \langle k_\alpha,x-y\rangle\in2\pi\mathbb Z
    \qquad (\alpha=1,\ldots,m).
\]
Since $\mathcal L$ is generated over $\mathbb Z$ by
$k_1,\ldots,k_m$, the same congruence holds for every $k\in\mathcal L$:
\[
    \langle k,x-y\rangle\in2\pi\mathbb Z
    \qquad\text{for all }k\in\mathcal L.
\]
By the definition of $\Lambda_{\mathcal L}$, this is exactly the statement that
$x-y\in\Lambda_{\mathcal L}$. Hence $x$ and $y$ represent the same point of
$\R^2/\Lambda_{\mathcal L}$, so the induced map is injective.

An injective immersion from a compact manifold is an embedding. More
explicitly, suppose that the same formula is considered on
$\R^2/\Gamma$, where $\Gamma\subset\Lambda_{\mathcal L}$ has finite index.
The natural quotient map
\[
    \pi:\R^2/\Gamma\longrightarrow \R^2/\Lambda_{\mathcal L}
\]
is a finite covering, and the map on $\R^2/\Gamma$ is the composition of $\pi$
with the map on $\R^2/\Lambda_{\mathcal L}$ constructed above. Since $d\pi$ is
an isomorphism at every point, the latter map is an immersion whenever its
lift to $\R^2/\Gamma$ is an immersion. This proves the final assertion.
\end{proof}

We shall use the following explicit computation of Lu's fundamental matrix for
the maps constructed in Lemma~\ref{lem:general-eigenmap}. This formula is the
main algebraic input in the construction in Proposition~\ref{prop:full-odd}.

\begin{lemma}[Lu's fundamental matrix for the flat construction]
\label{lem:flat-construction-matrix}
Let $F$ be the map in \eqref{eq:general-eigenmap}, and assume that
\eqref{eq:general-isometry} holds. Write
\[
    u_\alpha=(\cos\theta_\alpha,\sin\theta_\alpha),
    \qquad \alpha=1,\ldots,m,
\]
and set
\begin{equation}
\label{eq:R-def-general}
    R=\sum_{\alpha=1}^m w_\alpha e^{4i\theta_\alpha}.
\end{equation}
Then $S\equiv2$, and the two possible nonzero eigenvalues of Lu's fundamental matrix are
\begin{equation}
\label{eq:general-lambda}
    \lambda_1=1+|R|,
    \qquad
    \lambda_2=1-|R|.
\end{equation}
All remaining eigenvalues are zero. In particular,
\begin{equation}
\label{eq:general-Lu-value}
    S+\lambda_2=3-|R|,
\end{equation}
and
\begin{equation}
\label{eq:general-rho}
    \rho_0^\perp=4\lambda_1\lambda_2=4(1-|R|^2).
\end{equation}
\end{lemma}

\begin{proof}
We give a pointwise computation. The proof is written in orthonormal
coordinates on the domain, so that the formula for Lu's fundamental matrix from
\eqref{eq:fundamental-matrix} can be applied directly.

By the definition of the flat metric in \eqref{eq:flat-domain-metric}, the
change of variables
\begin{equation}
\label{eq:lem33-orthonormal-coordinates}
    y=\frac{L}{\sqrt2}x
\end{equation}
turns the domain metric into
\[
    g_L=dy_1^2+dy_2^2.
\]
Since $v_\alpha=Lu_\alpha$ and
$u_\alpha=(\cos\theta_\alpha,\sin\theta_\alpha)$, the phase appearing in
\eqref{eq:general-eigenmap} becomes
\begin{equation}
\label{eq:lem33-rescaled-phase}
    \langle v_\alpha,x\rangle
    =\left\langle \sqrt2\,u_\alpha,y\right\rangle.
\end{equation}
Thus, in the orthonormal coordinates $y_1,y_2$, the corresponding frequency
vector is
\begin{equation}
\label{eq:lem33-rescaled-frequency}
    k_\alpha=\sqrt2\,u_\alpha
    =\sqrt2(\cos\theta_\alpha,\sin\theta_\alpha).
\end{equation}
In particular, $|k_\alpha|^2=2$ for every $\alpha$.

For each $\alpha$, write
\begin{equation}
\label{eq:lem33-phase-definition}
    \phi_\alpha(y)=\langle k_\alpha,y\rangle.
\end{equation}
In the $\alpha$-th coordinate two-plane of $\R^{2m}$, let
\begin{equation}
\label{eq:lem33-radial-angular-vectors}
\begin{aligned}
    n_\alpha&=(0,\ldots,0,
    \cos\phi_\alpha,\sin\phi_\alpha,0,\ldots,0),\\
    m_\alpha&=(0,\ldots,0,
    -\sin\phi_\alpha,\cos\phi_\alpha,0,\ldots,0).
\end{aligned}
\end{equation}
Here $n_\alpha$ is the radial unit vector in that coordinate two-plane and
$m_\alpha$ is the angular unit vector. Formula \eqref{eq:general-eigenmap},
together with \eqref{eq:lem33-rescaled-phase}, becomes
\begin{equation}
\label{eq:lem33-F-radial-sum}
    F=\sum_{\alpha=1}^m \sqrt{w_\alpha}\,n_\alpha.
\end{equation}

Subscripts will always denote partial derivatives with respect to the
orthonormal variables $y_1,y_2$. From \eqref{eq:lem33-phase-definition} and
\eqref{eq:lem33-radial-angular-vectors}, we have
\[
    \frac{\partial n_\alpha}{\partial y_i}=k_{\alpha i}m_\alpha,
    \qquad
    \frac{\partial m_\alpha}{\partial y_j}=-k_{\alpha j}n_\alpha.
\]
Differentiating \eqref{eq:lem33-F-radial-sum} gives
\begin{equation}
\label{eq:lem33-first-second-derivatives}
\begin{aligned}
    F_i&=\sum_{\alpha=1}^m
    \sqrt{w_\alpha}\,k_{\alpha i}m_\alpha,\\
    F_{ij}&=-\sum_{\alpha=1}^m
    \sqrt{w_\alpha}\,k_{\alpha i}k_{\alpha j}n_\alpha.
\end{aligned}
\end{equation}
Because the coordinates $y_1,y_2$ are flat orthonormal coordinates on the
domain, there are no Christoffel-symbol terms in this computation.

We now pass from ordinary Euclidean second derivatives to the second fundamental
form in the unit sphere. For a surface in the unit sphere, the Euclidean second
derivative decomposes as
\begin{equation}
\label{eq:lem33-sphere-second-derivative}
    F_{ij}=h_{ij}-\delta_{ij}F,
\end{equation}
where $h_{ij}$ is the second fundamental form of the surface as a submanifold of
$\Sph^{2m-1}$. Equivalently, $h_{ij}=F_{ij}+\delta_{ij}F$. Combining
\eqref{eq:lem33-F-radial-sum} with \eqref{eq:lem33-first-second-derivatives}, we
obtain
\begin{equation}
\label{eq:hij-radial-general}
    h_{ij}=\sum_{\alpha=1}^m
    \sqrt{w_\alpha}(\delta_{ij}-k_{\alpha i}k_{\alpha j})n_\alpha.
\end{equation}

Let us check explicitly that the vector in \eqref{eq:hij-radial-general} is a
normal vector to the surface inside the sphere. First, using the first identity
in \eqref{eq:general-isometry} and the metric identity already used in
\eqref{eq:general-induced-metric-computation}, we get
\begin{equation}
\label{eq:lem33-orthogonal-to-radius}
\begin{aligned}
    \langle h_{ij},F\rangle
    &=\sum_{\alpha=1}^m
    w_\alpha(\delta_{ij}-k_{\alpha i}k_{\alpha j}) \\
    &=\delta_{ij}-\sum_{\alpha=1}^m
    w_\alpha k_{\alpha i}k_{\alpha j}
    =0.
\end{aligned}
\end{equation}
Here the last equality follows from \eqref{eq:lem33-rescaled-frequency} and the
second identity in \eqref{eq:general-isometry}. Second, each vector $F_k$ in
\eqref{eq:lem33-first-second-derivatives} is a linear combination of the angular
vectors $m_\alpha$, whereas each vector $h_{ij}$ in
\eqref{eq:hij-radial-general} is a linear combination of the radial vectors
$n_\alpha$. The radial vectors are orthogonal to all angular vectors, so
\begin{equation}
\label{eq:lem33-orthogonal-to-tangent}
    \langle h_{ij},F_k\rangle=0
    \qquad (i,j,k=1,2).
\end{equation}
Equations \eqref{eq:lem33-orthogonal-to-radius} and
\eqref{eq:lem33-orthogonal-to-tangent} show that
\eqref{eq:hij-radial-general} is the spherical second fundamental form.

There is one small point which is useful later. A single radial vector
$n_\alpha$ is not tangent to the ambient sphere, because
\begin{equation}
\label{eq:lem33-radial-not-spherical}
    \langle n_\alpha,F\rangle=\sqrt{w_\alpha}.
\end{equation}
Its projection to the tangent space of the unit sphere at $F$ is
\begin{equation}
\label{eq:lem33-spherical-projection}
    \nu_\alpha=n_\alpha-\sqrt{w_\alpha}F.
\end{equation}
Since every $F_i$ is a linear combination of the angular vectors $m_\beta$,
we also have $\langle\nu_\alpha,F_i\rangle=0$. Thus each $\nu_\alpha$ is a
normal vector to the surface inside the sphere.
If
\[
    c_{ij,\alpha}=\sqrt{w_\alpha}(\delta_{ij}-k_{\alpha i}k_{\alpha j}),
\]
then, by the same computation as in \eqref{eq:lem33-orthogonal-to-radius},
\begin{equation}
\label{eq:lem33-cancellation-hij}
    \sum_{\alpha=1}^m c_{ij,\alpha}\sqrt{w_\alpha}=0.
\end{equation}
Therefore the radial expression and the projected spherical expression agree:
\begin{equation}
\label{eq:lem33-radial-projected-agree}
    \sum_{\alpha=1}^m c_{ij,\alpha}n_\alpha
    =\sum_{\alpha=1}^m c_{ij,\alpha}\nu_\alpha.
\end{equation}
Moreover, from \eqref{eq:lem33-spherical-projection} one computes
\begin{equation}
\label{eq:lem33-projected-inner-basic}
    \langle\nu_\alpha,\nu_\beta\rangle
    =\delta_{\alpha\beta}-\sqrt{w_\alpha w_\beta}.
\end{equation}
Consequently, if two collections of coefficients $c_\alpha$ and $d_\alpha$
satisfy
\[
    \sum_{\alpha=1}^m c_\alpha\sqrt{w_\alpha}=0,
    \qquad
    \sum_{\alpha=1}^m d_\alpha\sqrt{w_\alpha}=0,
\]
then \eqref{eq:lem33-projected-inner-basic} gives
\begin{equation}
\label{eq:projected-inner-products}
    \left\langle
    \sum_{\alpha=1}^m c_\alpha\nu_\alpha,
    \sum_{\beta=1}^m d_\beta\nu_\beta
    \right\rangle
    =\sum_{\alpha=1}^m c_\alpha d_\alpha.
\end{equation}
Thus, for the coefficient vectors satisfying the cancellation condition, inner
products can be computed as if the projected vectors $\nu_\alpha$ were the
orthonormal radial vectors $n_\alpha$.

We now compute the two normal vectors which determine the nonzero part of Lu's
fundamental matrix. Using \eqref{eq:lem33-rescaled-frequency} in
\eqref{eq:hij-radial-general}, we obtain
\begin{equation}
\label{eq:lem33-h-components}
\begin{aligned}
    h_{11}
    &=-\sum_{\alpha=1}^m
    \sqrt{w_\alpha}\cos(2\theta_\alpha)\,\nu_\alpha,\\
    h_{12}
    &=-\sum_{\alpha=1}^m
    \sqrt{w_\alpha}\sin(2\theta_\alpha)\,\nu_\alpha,\\
    h_{22}&=-h_{11}.
\end{aligned}
\end{equation}
For example, the coefficient of $\nu_\alpha$ in $h_{11}$ is
$\sqrt{w_\alpha}(1-2\cos^2\theta_\alpha)=-\sqrt{w_\alpha}\cos(2\theta_\alpha)$,
and the coefficient in $h_{12}$ is
$-2\sqrt{w_\alpha}\cos\theta_\alpha\sin\theta_\alpha
=-\sqrt{w_\alpha}\sin(2\theta_\alpha)$.
The identities in \eqref{eq:general-isometry} are equivalent to
\begin{equation}
\label{eq:lem33-double-angle-balance}
    \sum_{\alpha=1}^m w_\alpha\cos(2\theta_\alpha)=0,
    \qquad
    \sum_{\alpha=1}^m w_\alpha\sin(2\theta_\alpha)=0.
\end{equation}
Thus the coefficient collections occurring in $h_{11}$ and $h_{12}$ satisfy the
cancellation hypothesis in \eqref{eq:projected-inner-products}. Applying
\eqref{eq:projected-inner-products} to \eqref{eq:lem33-h-components} gives
\begin{equation}
\label{eq:lem33-gram-entries}
\begin{aligned}
    |h_{11}|^2
    &=\sum_{\alpha=1}^m w_\alpha\cos^2(2\theta_\alpha),\\
    |h_{12}|^2
    &=\sum_{\alpha=1}^m w_\alpha\sin^2(2\theta_\alpha),\\
    \langle h_{11},h_{12}\rangle
    &=\sum_{\alpha=1}^m
    w_\alpha\cos(2\theta_\alpha)\sin(2\theta_\alpha).
\end{aligned}
\end{equation}
In particular,
\begin{equation}
\label{eq:trace-gram-one}
    |h_{11}|^2+|h_{12}|^2=\sum_{\alpha=1}^m w_\alpha=1.
\end{equation}
This also gives
\begin{equation}
\label{eq:lem33-S-equals-two}
    S=2|h_{11}|^2+2|h_{12}|^2=2.
\end{equation}

It remains to translate this Gram-matrix computation into the eigenvalues of
Lu's fundamental matrix. The minimality identity $h_{22}=-h_{11}$ shows that
the two normal vectors $h_{11}$ and $h_{12}$ determine the whole second
fundamental form. Applying the rank-two formula
\eqref{eq:fundamental-matrix} with these two vectors, the two possible nonzero eigenvalues
of $\mathcal A$ are precisely the eigenvalues of $2G$, where
\begin{equation}
\label{eq:lem33-two-by-two-operator}
    \qquad
    G=
    \begin{pmatrix}
    |h_{11}|^2 & \langle h_{11},h_{12}\rangle\\
    \langle h_{11},h_{12}\rangle & |h_{12}|^2
    \end{pmatrix}.
\end{equation}
All remaining eigenvalues are zero. This rank-two reduction is independent
of the choice of normal frame and automatically accounts for every normal
direction.

The entries of $G$ are controlled by the complex number $R$ in
\eqref{eq:R-def-general}. From \eqref{eq:lem33-gram-entries},
\begin{equation}
\label{eq:lem33-R-components}
\begin{aligned}
    |h_{11}|^2-|h_{12}|^2
    &=\sum_{\alpha=1}^m w_\alpha\cos(4\theta_\alpha)
    =\operatorname{Re}R,\\
    2\langle h_{11},h_{12}\rangle
    &=\sum_{\alpha=1}^m w_\alpha\sin(4\theta_\alpha)
    =\operatorname{Im}R.
\end{aligned}
\end{equation}
Together with \eqref{eq:trace-gram-one}, this gives the two eigenvalues of $G$:
\begin{equation}
\label{eq:lem33-G-eigenvalues}
    \frac{1+|R|}{2},
    \qquad
    \frac{1-|R|}{2}.
\end{equation}
Since $R$ is a convex combination of unit complex numbers, $|R|\le1$; hence
the two numbers in \eqref{eq:lem33-G-eigenvalues} are nonnegative and are
listed in decreasing order.
Since Lu's fundamental matrix has its possible nonzero eigenvalues equal to the
eigenvalues of $2G$ by \eqref{eq:lem33-two-by-two-operator}, we obtain
\eqref{eq:general-lambda}. Combining \eqref{eq:lem33-S-equals-two} with
\eqref{eq:general-lambda} gives \eqref{eq:general-Lu-value}.

Finally, the normal scalar curvature identity recalled in the preliminaries gives
\begin{equation}
\label{eq:lem33-rho0-computation}
    \rho_0^\perp
    =4\lambda_1\lambda_2
    =4(1+|R|)(1-|R|)
    =4(1-|R|^2),
\end{equation}
which is \eqref{eq:general-rho}. With the normalized convention
\eqref{eq:rho-normalized} for surfaces, this is equivalently
\[
    \rho^\perp=\sqrt{1-|R|^2}.
\]
The proof is complete.
\end{proof}

\subsection{The explicit codimension-three family}
\label{sec:explicit-q3}

We now record the codimension-three block used below in the density theorem and
in the non-fullness argument. Compared with the general construction,
this part gives completely explicit formulas. The construction is first written on a convenient finite-covering torus;
passing to the maximal frequency quotient gives the embedded torus with the
same local invariants.

\begin{theorem}[Explicit codimension-three finite-covering family]
\label{thm:explicit-q3}
For every pair of integers $N\ge2$ and $1\le\ell<N$, there is a closed linearly
full flat minimal immersion on an explicitly chosen finite-covering torus
\[
    F_{N,\ell}:T_{N,\ell}\longrightarrow\Sph^5
\]
with
\[
    S\equiv2,
    \qquad
    \lambda_1\equiv 2-\frac{\ell^2}{N^2},
    \qquad
    \lambda_2\equiv \frac{\ell^2}{N^2}.
\]
Thus
\[
    S+\lambda_2\equiv2+\frac{\ell^2}{N^2}.
\]
Moreover, $F_{N,\ell}$ factors through its maximal frequency quotient as a
closed linearly full flat minimal embedding
\[
    \widehat F_{N,\ell}:\widehat T_{N,\ell}^2\longrightarrow\Sph^5,
\]
and the same local identities for $S,\lambda_1,\lambda_2$ hold on this embedded
quotient. As $\ell/N$ ranges over rational numbers in $(0,1)$, the resulting
values contain
\[
    \{2+r^2:\ r\in\Q,\ 0<r<1\},
\]
and hence are dense in $(2,3)$.
\end{theorem}

\begin{proof}
Fix integers
\[
    N\ge2,
    \qquad
    1\le \ell<N.
\]
Choose
\[
    p=\left(\frac{\sqrt2}{N},0\right),
    \qquad
    q=\left(-\frac{\ell}{\sqrt2 N^2},
    \frac{\sqrt{4N^2-\ell^2}}{\sqrt2 N^2}\right).
\]
Then
\[
    |p|^2=|q|^2=\frac2{N^2},
    \qquad
    \ip{p}{q}=-\frac{\ell}{N^3}.
\]
We record the Gram matrix of the frequency basis $p,q$ in order to verify
Bryant's rationality criterion for the frequency lattice generated by these two
vectors. This Gram matrix is
\[
    G_{N,\ell}^*
    =
    \begin{pmatrix}
    \ip{p}{p} & \ip{p}{q}\\
    \ip{q}{p} & \ip{q}{q}
    \end{pmatrix}
    =\begin{pmatrix}
    2/N^2 & -\ell/N^3\\
    -\ell/N^3 & 2/N^2
    \end{pmatrix}
    \in M_2(\Q).
\]
In Bryant's notation, if $\alpha=p_1+ip_2$ and $\beta=q_1+iq_2$, then the two
rationality conditions can be read directly from the entries of this Gram
matrix. Indeed,
\[
    \frac{|\alpha|^2}{|\beta|^2}
    =
    \frac{|p|^2}{|q|^2}
    =1\in\Q,
\]
and
\[
    \operatorname{Re}\left(\frac{\alpha}{\beta}\right)
    =
    \frac{\operatorname{Re}(\alpha\overline{\beta})}{|\beta|^2}
    =
    \frac{\ip{p}{q}}{|q|^2}
    =
    \frac{-\ell/N^3}{2/N^2}
    =
    -\frac{\ell}{2N}
    \in\Q.
\]
Thus the frequency lattice generated by $p$ and $q$ satisfies Bryant's
rationality criterion. We now define the corresponding period lattice. Let
\[
    \Lambda_{N,\ell}=\{x\in\R^2:\ip{x}{p},\ip{x}{q}\in2\pi\Z\},
    \qquad
    T_{N,\ell}=\R^2/\Lambda_{N,\ell}.
\]
Let $B=(p\ \ q)$ be the $2\times2$ matrix whose columns are $p$ and $q$.
Since $p$ and $q$ are linearly independent, $B$ is invertible. Moreover, for
$x\in\R^2$,
\[
    B^T x=
    \begin{pmatrix}
        \ip{p}{x}\\
        \ip{q}{x}
    \end{pmatrix}.
\]
It follows directly from the definition of $\Lambda_{N,\ell}$ that
\[
\begin{aligned}
    x\in\Lambda_{N,\ell}
    &\iff B^T x\in2\pi\Z^2\\
    &\iff x\in2\pi(B^T)^{-1}\Z^2.
\end{aligned}
\]
Writing
\[
    B^{-T}:=(B^T)^{-1}=(B^{-1})^T,
\]
we therefore obtain
\begin{equation}
\label{eq:dual-lattice-basis}
\begin{aligned}
    \Lambda_{N,\ell}
    &=2\pi B^{-T}\Z^2
    =\left\{2\pi B^{-T}
    \begin{pmatrix}n_1\\ n_2\end{pmatrix}:n_1,n_2\in\Z\right\}.
\end{aligned}
\end{equation}
In particular, the two columns of $2\pi B^{-T}$ form a $\Z$-basis of the
period lattice $\Lambda_{N,\ell}$.

We next identify its $2\pi$-dual lattice. If
$\lambda=2\pi B^{-T}n$ with $n\in\Z^2$, then, for every $\xi\in\R^2$,
\[
    \ip{\xi}{\lambda}
    =2\pi\ip{B^{-1}\xi}{n}.
\]
Hence $\ip{\xi}{\lambda}\in2\pi\Z$ for every
$\lambda\in\Lambda_{N,\ell}$ if and only if
$B^{-1}\xi\in\Z^2$. Consequently,
\[
    \Lambda_{N,\ell}^*
    :=\{\xi\in\R^2:\ip{\xi}{\lambda}\in2\pi\Z
    \text{ for every }\lambda\in\Lambda_{N,\ell}\}
    =B\Z^2=\Z p\oplus\Z q.
\]
Thus $p$ and $q$ form a $\Z$-basis of $\Lambda_{N,\ell}^*$ in the
$2\pi$ normalization. Equivalently, a vector $\xi$ gives a well-defined
character $e^{i\ip{\xi}{x}}$ on $T_{N,\ell}$ if and only if
$\xi\in\Lambda_{N,\ell}^*$. Hence every integer linear combination of $p$
and $q$ gives a well-defined character on $T_{N,\ell}$.

Define
\[
    k_1=Np,
    \qquad
    k_2=Nq,
    \qquad
    k_3=Np+\ell q.
\]
Then
\[
    |k_1|^2=|k_2|^2=|k_3|^2=2.
\]
Indeed,
\[
    |k_3|^2=N^2|p|^2+2N\ell\ip{p}{q}+\ell^2|q|^2
    =2-\frac{2\ell^2}{N^2}+\frac{2\ell^2}{N^2}=2.
\]
Put
\begin{equation}
\label{eq:weights}
    w_1=w_3=\frac{N^2}{4N^2-\ell^2},
    \qquad
    w_2=\frac{2N^2-\ell^2}{4N^2-\ell^2}.
\end{equation}
To verify the metric identity without a lengthy polynomial expansion, introduce
\[
    \varepsilon=\frac{\ell}{2N},
    \qquad
    \zeta=\frac{\sqrt{4N^2-\ell^2}}{2N}.
\]
Then
\[
    \varepsilon^2+\zeta^2=1,
    \qquad 0<\varepsilon<\frac12,
\]
and the unit directions $u_\alpha=k_\alpha/\sqrt2$ are
\[
    u_1=(1,0),\qquad
    u_2=(-\varepsilon,\zeta),\qquad
    u_3=(\zeta^2-\varepsilon^2,2\varepsilon\zeta).
\]
The weights in \eqref{eq:weights} become
\[
    w_1=w_3=\frac1{4\zeta^2},
    \qquad
    w_2=\frac{\zeta^2-\varepsilon^2}{2\zeta^2}.
\]
Since
$\zeta^2-\varepsilon^2=1-2\varepsilon^2>0$, all three weights are positive.
Moreover,
\[
    w_1+w_2+w_3
    =\frac{1+\zeta^2-\varepsilon^2}{2\zeta^2}=1.
\]
Let
\[
    M=\sum_{\alpha=1}^3w_\alpha u_\alpha\otimes u_\alpha.
\]
The $(2,2)$ and $(1,2)$ entries are
\[
\begin{aligned}
    M_{22}
    &=w_2\zeta^2+4w_3\varepsilon^2\zeta^2
      =\frac{\zeta^2-\varepsilon^2}{2}+\varepsilon^2
      =\frac12,\\
    M_{12}
    &=-w_2\varepsilon\zeta
      +2w_3\varepsilon\zeta(\zeta^2-\varepsilon^2)
      =0.
\end{aligned}
\]
Moreover,
\[
    \operatorname{Tr}M=\sum_{\alpha=1}^3w_\alpha|u_\alpha|^2
    =\sum_{\alpha=1}^3w_\alpha=1,
\]
so $M_{11}=1-M_{22}=1/2$. Hence
\[
    \sum_{\alpha=1}^3w_\alpha u_\alpha\otimes u_\alpha=\frac12I_2.
\]
Since $k_\alpha=\sqrt2u_\alpha$, this is equivalent to
\begin{equation}
\label{eq:isometry-condition}
    \sum_{\alpha=1}^3w_\alpha k_\alpha\otimes k_\alpha=I_2.
\end{equation}
Equivalently, \eqref{eq:isometry-condition} says that the induced metric of the
map below is the Euclidean metric $dx_1^2+dx_2^2$.
Thus the map
\[
\begin{aligned}
F_{N,\ell}(x)=(&\sqrt{w_1}\cos\ip{k_1}{x},
\sqrt{w_1}\sin\ip{k_1}{x},
\sqrt{w_2}\cos\ip{k_2}{x}, \\
&\sqrt{w_2}\sin\ip{k_2}{x},
\sqrt{w_3}\cos\ip{k_3}{x},
\sqrt{w_3}\sin\ip{k_3}{x})
\end{aligned}
\]
is well-defined on $T_{N,\ell}$, because $k_1,k_2,k_3\in\Lambda_{N,\ell}^*$.
Moreover $|F_{N,\ell}|^2=w_1+w_2+w_3=1$, and
\eqref{eq:isometry-condition} gives
\[
    F_{N,\ell}^*g_{\Sph^5}=dx_1^2+dx_2^2.
\]
Since $|k_1|^2=|k_2|^2=|k_3|^2=2$, each coordinate function is a $2$-eigenfunction
for the positive operator $-\Delta$; equivalently, with our convention,
\[
    \Delta F_{N,\ell}+2F_{N,\ell}=0.
\]
Takahashi's theorem therefore gives minimality. The immersion is linearly full.
Indeed, the three frequencies are
\[
    k_1=Np,\qquad k_2=Nq,\qquad k_3=Np+\ell q.
\]
Thus, when they are written in the basis $p,q$ of the dual lattice
$\Lambda_{N,\ell}^*$, their coefficient vectors are
\[
    (N,0),\qquad (0,N),\qquad (N,\ell).
\]
Since $0<\ell<N$, these coefficient vectors are pairwise distinct and none is
the negative of another. Hence the frequencies $k_1,k_2,k_3$ are pairwise
distinct up to sign. Distinct characters of a compact torus are orthogonal in
$L^2$. Therefore the six real functions
\[
    \cos\ip{k_\alpha}{x},\quad \sin\ip{k_\alpha}{x},
    \qquad \alpha=1,2,3,
\]
are linearly independent. Since all weights are positive, no coordinate pair is
absent. Hence the image is not contained in any proper linear subspace of
$\mathbb R^6$, and the immersion is linearly full in $\Sph^5$. By
Proposition~\ref{prop:embedded-quotient}, the same image is obtained from an embedded flat
minimal torus after quotienting by the full period lattice generated by
$k_1,k_2,k_3$. The local invariants computed below are unchanged by this finite
quotient, and linear fullness is preserved because the image is unchanged.

We now compute the complex number $R_{N,\ell}$ from
Lemma~\ref{lem:flat-construction-matrix}. Choose
$\gamma\in(0,\pi/6)$ such that
\[
    \varepsilon=\sin\gamma,
    \qquad
    \zeta=\cos\gamma.
\]
The three unit directions above have angles
\[
    \theta_1=0,
    \qquad
    \theta_2=\frac\pi2+\gamma,
    \qquad
    \theta_3=2\gamma.
\]
Indeed,
$u_2=(-\sin\gamma,\cos\gamma)$ and
$u_3=(\cos2\gamma,\sin2\gamma)$. Since
$e^{4i(\pi/2+\gamma)}=e^{4i\gamma}$ and $w_1=w_3$, we obtain
\begin{align*}
    R_{N,\ell}
    &=w_1+w_2e^{4i\gamma}+w_3e^{8i\gamma}\\
    &=e^{4i\gamma}\bigl(w_2+2w_1\cos4\gamma\bigr)\\
    &=e^{4i\gamma}
      \frac{\zeta^2-\varepsilon^2+\cos4\gamma}{2\zeta^2}.
\end{align*}
Using $\cos4\gamma=1-8\varepsilon^2\zeta^2$ and
$\zeta^2=1-\varepsilon^2$, the real factor simplifies to
\[
    \frac{\zeta^2-\varepsilon^2+\cos4\gamma}{2\zeta^2}
    =1-4\varepsilon^2
    =1-\frac{\ell^2}{N^2}.
\]
This factor is positive because $0<\ell<N$. Therefore
\begin{equation}
\label{eq:R-abs}
    |R_{N,\ell}|=1-\frac{\ell^2}{N^2}.
\end{equation}
Consequently
\[
    \lambda_1=2-\frac{\ell^2}{N^2},
    \qquad
    \lambda_2=\frac{\ell^2}{N^2}.
\]
Since $S\equiv2$, this gives
\begin{equation}
\label{eq:Lu-value}
    S+\lambda_2=2+\frac{\ell^2}{N^2}.
\end{equation}
Finally,
\[
    \rho_0^\perp=4\lambda_1\lambda_2
    =4\frac{\ell^2}{N^2}\left(2-\frac{\ell^2}{N^2}\right).
\]
By \eqref{eq:rho-normalized}, one has

\[
    \rho^\perp=\frac{\ell}{N}\sqrt{2-\frac{\ell^2}{N^2}}.
\]
For fixed $\ell$, these quantities satisfy
\[
    S+\lambda_2\downarrow2,
    \qquad
    \rho_0^\perp\downarrow0,
    \qquad
    N\to\infty.
\]
Moreover, because $\ell/N$ can be any rational number in $(0,1)$, the realized values $2+\ell^2/N^2$ contain $\{2+r^2:\ r\in\Q,\ 0<r<1\}$ and are dense in $(2,3)$.
The embedded quotient construction and the preservation of the local invariants were established
above using Proposition~\ref{prop:embedded-quotient}; this completes the proof.
\end{proof}

\begin{proposition}[Area growth of the explicit family]
\label{prop:area-growth-explicit}
For the explicit finite-covering tori $T_{N,\ell}$ constructed above,
\begin{equation}
\label{eq:area}
    \Area(T_{N,\ell})
    =\frac{4\pi^2N^3}{\sqrt{4N^2-\ell^2}}.
\end{equation}
Put
\[
    d:=\gcd(N,\ell),
\]
where $\gcd$ denotes the greatest common divisor. Then the covering degree from $T_{N,\ell}$ to the maximal embedded frequency quotient
$\widehat T_{N,\ell}$ is $Nd$, and hence
\begin{equation}
\label{eq:embedded-quotient-area}
    \Area(\widehat T_{N,\ell})
    =\frac{4\pi^2N^2}{d\sqrt{4N^2-\ell^2}}.
\end{equation}
Consequently, for each fixed $\ell$,
\[
    \Area(T_{N,\ell})\sim 2\pi^2N^2
    \qquad\text{and}\qquad
    (S+\lambda_2-2)\Area(T_{N,\ell})\to 2\pi^2\ell^2
\]
as $N\to\infty$. Thus the reciprocal-area scale asserted here is measured on
the explicitly parametrized finite-covering torus. On the maximal embedded
quotient the local quantities $S$, $\lambda_1$, and $\lambda_2$ are unchanged,
but the area is smaller by the covering degree $Nd$.
\end{proposition}

\begin{proof}
The covolume of the lattice generated by the dual frequency basis $p,q$ is
\[
    |\det(p,q)|
    =\frac{\sqrt2}{N}\,
      \frac{\sqrt{4N^2-\ell^2}}{\sqrt2N^2}
    =\frac{\sqrt{4N^2-\ell^2}}{N^3}.
\]
Here $\det(p,q)$ denotes the determinant of the $2\times2$ matrix whose columns
are $p$ and $q$. Since $\Lambda_{N,\ell}=2\pi B^{-T}\Z^2$, where $B$ has
columns $p,q$, its covolume is $(2\pi)^2/|\det B|$. The induced metric on
$T_{N,\ell}$ is Euclidean, so its area equals this covolume. Thus
\[
    \Area(T_{N,\ell})
    =\frac{(2\pi)^2}{|\det(p,q)|}
    =\frac{4\pi^2N^3}{\sqrt{4N^2-\ell^2}}.
\]
The full frequency lattice generated by $k_1,k_2,k_3$ is
\[
\begin{aligned}
    \mathcal L_{N,\ell}
    &=\langle Np,Nq,Np+\ell q\rangle_{\mathbb Z}\\
    &=\langle Np,Nq,\ell q\rangle_{\mathbb Z}
      =\langle Np,dq\rangle_{\mathbb Z},
      \qquad d:=\gcd(N,\ell).
\end{aligned}
\]
Here $d$ is the greatest common divisor of $N$ and $\ell$. The last equality
follows from B\'ezout's identity: $d$ is an integral linear combination of $N$
and $\ell$, while $N$ and $\ell$ are both multiples of $d$.
Relative to the basis $p,q$ of $\Lambda_{N,\ell}^*$, this says that the
coefficient lattice of $\mathcal L_{N,\ell}$ is
\[
    N\Z\oplus d\Z\subset\Z^2.
\]
To see the corresponding period lattice and its index explicitly, let
\[
    D=\begin{pmatrix}N&0\\0&d\end{pmatrix}.
\]
Since the columns of $B$ are $p,q$, the matrix $BD$ has columns $Np$ and $dq$;
hence
\[
    \mathcal L_{N,\ell}=BD\Z^2.
\]
By the same period-lattice convention used in \eqref{eq:dual-lattice-basis}, the
full period lattice of the three frequencies is
\[
\begin{aligned}
    \widehat\Lambda_{N,\ell}
    &=2\pi(BD)^{-T}\Z^2\\
    &=2\pi B^{-T}D^{-T}\Z^2.
\end{aligned}
\]
Here
\[
    D^{-T}=D^{-1}=\begin{pmatrix}1/N&0\\0&1/d\end{pmatrix},
    \qquad
    D^{-T}\Z^2=\frac1N\Z\oplus\frac1d\Z.
\]
By contrast, \eqref{eq:dual-lattice-basis} gives
$\Lambda_{N,\ell}=2\pi B^{-T}\Z^2$. Since
\[
    \Z^2\subset D^{-T}\Z^2
\]
with index $Nd$, multiplication by the invertible matrix $2\pi B^{-T}$ gives
\[
    \Lambda_{N,\ell}\subset\widehat\Lambda_{N,\ell},
    \qquad
    [\widehat\Lambda_{N,\ell}:\Lambda_{N,\ell}]=Nd.
\]
Thus the quotient map
$\R^2/\Lambda_{N,\ell}\to\R^2/\widehat\Lambda_{N,\ell}$ has degree $Nd$.
It divides the area by $Nd$ and gives \eqref{eq:embedded-quotient-area}. For
fixed $\ell$, \eqref{eq:area} gives
\[
    \Area(T_{N,\ell})\sim 2\pi^2N^2.
\]
Together with \eqref{eq:Lu-value}, namely
\[
    S+\lambda_2-2=\frac{\ell^2}{N^2},
\]
we obtain
\begin{equation}
\label{eq:sharp-scale}
    (S+\lambda_2-2)\Area(T_{N,\ell})
    =\frac{4\pi^2\ell^2N}{\sqrt{4N^2-\ell^2}}
    \longrightarrow 2\pi^2\ell^2.
\end{equation}
This proves the proposition.
\end{proof}

\begin{remark}[Area scale for the finite-covering parametrizations]
\label{rem:sharp-finite-cover-scale}
The family in Theorem~\ref{thm:explicit-q3} is written on explicitly chosen
finite-covering domains $T_{N,\ell}$. On these domains,
\[
    \Area(T_{N,\ell})=
    \frac{4\pi^2N^3}{\sqrt{4N^2-\ell^2}}
    \sim 2\pi^2N^2,
\]
and
\[
    \bigl(S+\lambda_2-2\bigr)\Area(T_{N,\ell})
    \to2\pi^2\ell^2.
\]
Thus the gap $S+\lambda_2-2$ is of order the reciprocal of the area of these
chosen finite-covering domains. If
\[
    d:=\gcd(N,\ell),
\]
where $\gcd$ denotes the greatest common divisor, then the maximal embedded quotient has
\[
    \Area(\widehat T_{N,\ell})
    =\frac{4\pi^2N^2}{d\sqrt{4N^2-\ell^2}},
\]
while the local quantities $S,\lambda_1,\lambda_2$ are unchanged. In particular,
for fixed $\ell$,
\[
    (S+\lambda_2-2)\Area(\widehat T_{N,\ell})
    =\frac{4\pi^2\ell^2}{d\sqrt{4N^2-\ell^2}}
    \longrightarrow0.
\]
Accordingly, the reciprocal-area scale displayed above is a scale measured on
the explicit finite-covering parametrizations; it is not meant as a universal area
law for every embedded quotient of the same image.
\end{remark}

\subsection{Odd-codimensional density and sequences}
\label{sec:odd-density}

We next show that such data exist for every $m\ge3$ and can be chosen so that Lu's quantity approaches the first-gap value. The construction uses rational points on the unit circle. If $t\in\Q$, then
\begin{equation}
\label{eq:rational-circle}
    u(t)=\left(\frac{1-t^2}{1+t^2},\frac{2t}{1+t^2}\right)
\end{equation}
is a rational unit vector. We shall use the following elementary observation
several times. Suppose that $u_1,\ldots,u_m\in\R^2$ are unit vectors with
rational coordinates. Choose a positive integer $L$ such that
$Lu_\alpha\in\Z^2$ for every $\alpha$, and set
\[
    v_\alpha:=Lu_\alpha.
\]
Then
\[
    v_\alpha\in\Z^2,\qquad |v_\alpha|=L
\]
for every $\alpha$, since $|u_\alpha|=1$. Hence all frequencies have the same
Euclidean length. Moreover, because $v_\alpha$ has integer coordinates, the
functions
\[
    e^{i\langle v_\alpha,x\rangle},\qquad
    \cos\langle v_\alpha,x\rangle,\qquad
    \sin\langle v_\alpha,x\rangle
\]
are well-defined on the standard flat torus $(\R/2\pi\Z)^2$.

\begin{proposition}[The general odd-codimensional sequence]
\label{prop:full-odd}
For every $m\ge3$ there exist rational unit directions
\[
    u_{1,j},\ldots,u_{m,j}\in\mathbb S^1\subset\R^2,
    \qquad j=1,2,\ldots,
\]
positive rational weights $w_{1,j},\ldots,w_{m,j}$, and positive integers
$L_j$ such that
\begin{equation}
\label{eq:prop33-conditions}
    \sum_{\alpha=1}^m w_{\alpha,j}=1,
    \qquad
    \sum_{\alpha=1}^m w_{\alpha,j}
    u_{\alpha,j}\otimes u_{\alpha,j}=\frac12 I_2,
\end{equation}
and such that
\[
    v_{\alpha,j}:=L_j u_{\alpha,j}\in\Z^2,
    \qquad |v_{\alpha,j}|=L_j,
    \qquad \alpha=1,\ldots,m.
\]
The directions $u_{1,j},\ldots,u_{m,j}$ may be chosen pairwise distinct up to
sign. Let
\[
    F_j:T_{L_j}^2\longrightarrow\Sph^{2m-1}
\]
    be the immersion of Lemma~\ref{lem:general-eigenmap} determined by these
vectors and weights. Then $F_j$ is a linearly full flat minimal immersion. Moreover,
$F_j$ factors through a linearly full flat minimal embedding on the maximal
frequency quotient,
\[
    \widehat F_j:\widehat T_j^2\longrightarrow\Sph^{2m-1}.
\]
The quotient has the same image as $F_j$, so linear fullness is preserved, and
all local invariants are unchanged. In particular,
\[
    S\equiv2,
    \qquad
    S+\lambda_2>2,
    \qquad
    S+\lambda_2\longrightarrow2.
\]
Consequently, for every odd codimension $q=2m-3\ge3$, there are linearly full
flat minimal embeddings in $\Sph^{2+q}$ that contradict any positive
second-gap assertion in the class of closed embedded flat minimal tori, and
hence also the original closed immersed formulation.
\end{proposition}

\begin{proof}
We give the construction separately according to the parity of $m$, the number
of frequency pairs; the codimension $q=2m-3$ is odd throughout. The only
ingredients are rational points on the unit circle and the elementary identity
\[
    u\otimes u+u^\perp\otimes u^\perp=I_2,
    \qquad
    u^\perp=(-u_2,u_1),
\]
valid for every unit vector $u=(u_1,u_2)$.

\smallskip
\noindent\emph{Step 1: the case $m=2r$.}
Choose positive rational numbers $\rho_1,\ldots,\rho_r$ with
$\sum_{\kappa=1}^r\rho_\kappa=1$. For the $j$-th member of the sequence, choose
rational numbers
\[
    0=t_{1,j}<t_{2,j}<\cdots<t_{r,j},
    \qquad t_{\kappa,j}\to0,
\]
so small that all the corresponding angles lie in $[0,\pi/8)$. Set
\[
    u_{\kappa,j}=u(t_{\kappa,j}),
    \qquad
    u_{\kappa+r,j}=u_{\kappa,j}^{\perp},
    \qquad \kappa=1,\ldots,r,
\]
where $u(t)$ is the rational parametrization \eqref{eq:rational-circle}, and put
\[
    w_{\kappa,j}=w_{\kappa+r,j}=\frac{\rho_\kappa}{2}.
\]
The vectors $u_{\kappa,j}$ lie in a small arc around $(1,0)$. With our
convention $u^\perp=(-u_2,u_1)$, the vector $u^\perp$ is obtained from $u$ by a
counterclockwise rotation through $\pi/2$; hence the orthogonal vectors
$u_{\kappa,j}^\perp$ lie in the corresponding small arc around $(0,1)$. The strict
ordering of the parameters therefore makes all $2r$ directions pairwise
distinct up to sign.
Then
\[
    \sum_{\alpha=1}^{2r}w_{\alpha,j}=1,
\]
and, for each $\kappa$,
\[
    \frac{\rho_\kappa}{2}u_{\kappa,j}\otimes u_{\kappa,j}
    +\frac{\rho_\kappa}{2}u_{\kappa,j}^{\perp}\otimes u_{\kappa,j}^{\perp}
    =\frac{\rho_\kappa}{2}I_2.
\]
Summing over $\kappa$ gives the second identity in \eqref{eq:prop33-conditions}.
Since all coordinates of the $u_{\alpha,j}$ are rational, we can choose a common
denominator $L_j$ so that
\[
    v_{\alpha,j}=L_j u_{\alpha,j}\in\Z^2,
    \qquad |v_{\alpha,j}|=L_j.
\]
Thus all assumptions of Lemma~\ref{lem:general-eigenmap} are satisfied. The
pairwise distinctness up to sign gives linear fullness.

\smallskip
\noindent\emph{Step 2: the case $m=2r+1$.}
We first build a three-direction block. Choose rational numbers $\varepsilon_j,\zeta_j>0$ with
\[
    \varepsilon_j^2+\zeta_j^2=1,
    \qquad
    0<\varepsilon_j<\frac12,
    \qquad
    \varepsilon_j\to0.
\]
For example, take a sufficiently small decreasing rational sequence $s_j>0$
and put
\[
    \varepsilon_j=\frac{2s_j}{1+s_j^2},
    \qquad
    \zeta_j=\frac{1-s_j^2}{1+s_j^2},
    \qquad
    s_j\in\Q,
    \qquad s_j\downarrow0.
\]
Define three unit vectors in $\mathbb S^1\subset\R^2$ by
\[
    u_{0,j}=(1,0),
    \qquad
    u_{+,j}=(\varepsilon_j,\zeta_j),
    \qquad
    u_{-,j}=(-\varepsilon_j,\zeta_j),
\]
and define the corresponding scalar weights by
\begin{equation}
\label{eq:three-block-weights}
    \mu_{0,j}=\frac{\zeta_j^2-\varepsilon_j^2}{2\zeta_j^2},
    \qquad
    \mu_{+,j}=\mu_{-,j}=\frac{1}{4\zeta_j^2}.
\end{equation}
Since
$\zeta_j^2-\varepsilon_j^2=1-2\varepsilon_j^2>0$, all three weights are
positive. Their sum is
\[
    \mu_{0,j}+\mu_{+,j}+\mu_{-,j}
    =\frac{\zeta_j^2-\varepsilon_j^2+1}{2\zeta_j^2}=1.
\]
Moreover, the off-diagonal terms cancel by symmetry, while
\[
    \mu_{0,j}+2\mu_{+,j}\varepsilon_j^2=\frac12,
    \qquad
    2\mu_{+,j}\zeta_j^2=\frac12.
\]
Hence
\begin{equation}
\label{eq:three-block-moment}
    \mu_{0,j}u_{0,j}\otimes u_{0,j}
    +\mu_{+,j}u_{+,j}\otimes u_{+,j}
    +\mu_{-,j}u_{-,j}\otimes u_{-,j}
    =\frac12 I_2.
\end{equation}
If $r=1$, set
\[
    (u_{1,j},u_{2,j},u_{3,j})=(u_{0,j},u_{+,j},u_{-,j}),
    \qquad
    (w_{1,j},w_{2,j},w_{3,j})=(\mu_{0,j},\mu_{+,j},\mu_{-,j}).
\]
This gives the required $m=3$ directions and weights.

Assume next that $r\ge2$. Choose positive rational constants
$\eta_1,\ldots,\eta_{r-1}$ and set
\[
    \eta=\sum_{\sigma=1}^{r-1}\eta_\sigma<1.
\]
For each $\sigma$ choose a rational unit vector
\[
    p_{\sigma,j}=u(\tau_{\sigma,j}),
    \qquad \tau_{\sigma,j}\in\Q,
    \qquad \tau_{\sigma,j}\to0,
\]
with the parameters chosen so that all directions in the displayed block and all
added directions $p_{\sigma,j},p_{\sigma,j}^{\perp}$ are pairwise distinct up
to sign. Give the block directions
$u_{0,j},u_{+,j},u_{-,j}$ the respective weights
$(1-\eta)\mu_{0,j},(1-\eta)\mu_{+,j},(1-\eta)\mu_{-,j}$, and give each member
of the pair $p_{\sigma,j},p_{\sigma,j}^{\perp}$ weight $\eta_\sigma/2$.
Relabel the resulting weighted collection
\[
    u_{0,j},u_{+,j},u_{-,j},
    p_{1,j},p_{1,j}^{\perp},\ldots,
    p_{r-1,j},p_{r-1,j}^{\perp}
\]
as $(u_{1,j},w_{1,j}),\ldots,(u_{m,j},w_{m,j})$. Its total weight is
\[
    (1-\eta)+\sum_{\sigma=1}^{r-1}\eta_\sigma=1.
\]
Using $p\otimes p+p^\perp\otimes p^\perp=I_2$ together with
\eqref{eq:three-block-moment}, its weighted moment is
\begin{align*}
    &(1-\eta)
    \left(
    \mu_{0,j}u_{0,j}\otimes u_{0,j}
    +\mu_{+,j}u_{+,j}\otimes u_{+,j}
    +\mu_{-,j}u_{-,j}\otimes u_{-,j}
    \right)  \\
    &\qquad
    +\sum_{\sigma=1}^{r-1}\frac{\eta_\sigma}{2}
    \left(p_{\sigma,j}\otimes p_{\sigma,j}
    +p_{\sigma,j}^{\perp}\otimes p_{\sigma,j}^{\perp}\right) \\
    &=(1-\eta)\frac12 I_2
    +\sum_{\sigma=1}^{r-1}\frac{\eta_\sigma}{2}I_2
    =\frac12 I_2.
\end{align*}

In either case, the resulting $m$ directions are rational and pairwise distinct
up to sign, all weights are positive, and \eqref{eq:prop33-conditions} holds.
Choose a
common positive integer denominator $L_j$ and put
\[
    v_{\alpha,j}=L_j u_{\alpha,j}\in\Z^2,
    \qquad |v_{\alpha,j}|=L_j.
\]
Lemma~\ref{lem:general-eigenmap} therefore gives linearly full flat minimal
immersions with $S\equiv2$.

\smallskip
\noindent\emph{Step 3: the limiting value of $S+\lambda_2$.}
Write
\[
    u_{\alpha,j}=(\cos\theta_{\alpha,j},\sin\theta_{\alpha,j})
\]
and define
\[
    R_j=\sum_{\alpha=1}^m w_{\alpha,j}e^{4i\theta_{\alpha,j}}.
\]
Lemma~\ref{lem:flat-construction-matrix} gives
\[
    \lambda_1=1+|R_j|,
    \qquad
    \lambda_2=1-|R_j|,
    \qquad
    S+\lambda_2=3-|R_j|.
\]
It remains only to record the behavior of $R_j$ for the directions constructed
above.

If $m=2r$, the two vectors in each orthogonal pair give the same complex factor
in $R_j$, because
$e^{4i(\theta+\pi/2)}=e^{4i\theta}$. Therefore
\[
    R_j=\sum_{\kappa=1}^r\rho_\kappa e^{4i\theta_{\kappa,j}}.
\]
Since all $\theta_{\kappa,j}\to0$, we have $R_j\to1$. We also claim that
$|R_j|<1$ for every $j$. Indeed, each number $e^{4i\theta_{\kappa,j}}$ lies on the
unit circle, and the weights $\rho_\kappa$ are positive with
$\sum_{\kappa=1}^r\rho_\kappa=1$. Hence the triangle inequality gives
\[
    |R_j|
    =\left|\sum_{\kappa=1}^r\rho_\kappa e^{4i\theta_{\kappa,j}}\right|
    \le \sum_{\kappa=1}^r\rho_\kappa |e^{4i\theta_{\kappa,j}}|
    =1.
\]
Equality in this inequality could occur only if all the unit complex numbers
$e^{4i\theta_{\kappa,j}}$ point in the same direction. In our construction they are
not all equal: the angles $\theta_{\kappa,j}$ are distinct and lie in
$[0,\pi/8)$, on which $\theta\mapsto e^{4i\theta}$ is injective. Therefore the
inequality is strict, and $|R_j|<1$.

If $m=2r+1$, choose $\gamma_j\in(0,\pi/6)$ so that
\[
    \varepsilon_j=\sin\gamma_j,
    \qquad
    \zeta_j=\cos\gamma_j.
\]
Then $u_{+,j}$ and $u_{-,j}$ have angles $\pi/2-\gamma_j$ and
$\pi/2+\gamma_j$, respectively. Hence the three-direction block contributes
\[
\begin{aligned}
    R_j^{(3)}
    &=\mu_{0,j}
      +\mu_{+,j}e^{4i(\pi/2-\gamma_j)}
      +\mu_{-,j}e^{4i(\pi/2+\gamma_j)}\\
    &=\frac{\zeta_j^2-\varepsilon_j^2+\cos(4\gamma_j)}{2\zeta_j^2}\\
    &=1-4\varepsilon_j^2.
\end{aligned}
\]
For the last equality we used
$\cos(4\gamma_j)=1-8\varepsilon_j^2\zeta_j^2$ and
$\varepsilon_j^2+\zeta_j^2=1$.
Since $0<\varepsilon_j<1/2$, we have
$0<R_j^{(3)}<1$, and $R_j^{(3)}\to1$. If $m=3$, then
$R_j=R_j^{(3)}$. If $m>3$, the remaining orthogonal pairs contribute terms
$\eta_\sigma e^{4i\varphi_{\sigma,j}}$ with
$\varphi_{\sigma,j}\to0$, so
\[
    R_j=(1-\eta)R_j^{(3)}+
    \sum_{\sigma=1}^{r-1}\eta_\sigma e^{4i\varphi_{\sigma,j}}\longrightarrow1.
\]
The same argument gives $|R_j|<1$ in this case as well. In the ungrouped
definition of $R_j$, it is a convex combination of unit complex numbers with
positive coefficients. The three-direction block already contains the factors
$1$, $e^{-4i\gamma_j}$, and $e^{4i\gamma_j}$, which are not all equal because
$0<\gamma_j<\pi/6$. Thus equality in the triangle inequality is impossible.

In either parity, $R_j\to1$ and $|R_j|<1$. After discarding finitely many terms
and relabeling the sequence, we may also assume $|R_j|>0$ for every $j$.
Therefore $S+\lambda_2=3-|R_j|$ is strictly larger than $2$ and tends to $2$.
The positive definiteness of the moment identity in \eqref{eq:prop33-conditions} implies that
the directions $u_{\alpha,j}$ span $\R^2$. Hence the integer frequencies
$v_{\alpha,j}$ generate a rank-two frequency lattice.
Proposition~\ref{prop:embedded-quotient} therefore applies and shows that each map factors
through an embedded flat minimal torus on its maximal frequency quotient. Passing
to this quotient does not change the image or any local invariant.
\end{proof}

\begin{proof}[\textbf{Proof of Theorem~\ref{thm:main}}]
Fix an odd codimension $q=2m-3\ge3$, a target value $c\in(2,3)$, and
$\delta>0$. Put
\[
    t=3-c\in(0,1),
    \qquad
    \delta_0=\min\left\{\delta,\frac t2,\frac{1-t}{2}\right\}>0.
\]
It is enough to construct rational unit directions
$u_\alpha=(\cos\theta_\alpha,\sin\theta_\alpha)$, pairwise distinct up to sign,
and positive weights satisfying
\[
    \sum_\alpha w_\alpha u_\alpha\otimes u_\alpha=\frac12I_2
\]
and
\[
    \bigl||R|-t\bigr|<\delta_0,
    \qquad
    R=\sum_\alpha w_\alpha e^{4i\theta_\alpha}.
\]
Indeed, taking the trace of the moment identity gives $\sum_\alpha w_\alpha=1$, because each $u_\alpha$ is a unit vector.
Lemma~\ref{lem:flat-construction-matrix} then yields
$S+\lambda_2=3-|R|$. The choice of $\delta_0$ will also ensure
$0<|R|<1$, so the realized value belongs to $(2,3)$.

Suppose first that $m=2r$ is even. Since $m\ge3$, we have $r\ge2$. Take
$p_1=(1,0)$. For a unit vector $p=(\cos\theta,\sin\theta)$,
\[
    \left|\frac{1+e^{4i\theta}}2\right|=|\cos2\theta|.
\]
As $\theta$ runs through $[0,\pi/4]$, this expression runs through $[0,1]$.
Rational points are dense on the unit circle, so we may choose a rational unit
vector $p_2=(\cos\theta_2,\sin\theta_2)$, distinct up to sign from both
$p_1$ and $p_1^\perp$, such that
\[
    \left|\left|\frac{1+e^{4i\theta_2}}2\right|-t\right|
    <\frac{\delta_0}{4}.
\]
Assign weight $1/4$ to each of
$p_1,p_1^\perp,p_2,p_2^\perp$. Since
$p\otimes p+p^\perp\otimes p^\perp=I_2$, these four directions satisfy the
moment condition, and their complex number is
\[
    R_0=\frac12\left(1+e^{4i\theta_2}\right).
\]

If $r=2$, set $\eta=0$, $E=0$, and $R=R_0$. If $r>2$, choose rational unit vectors
$p_3,\ldots,p_r$ so that all directions
$p_1,p_1^\perp,\ldots,p_r,p_r^\perp$ are pairwise distinct up to sign. Choose
positive rational numbers $\eta_3,\ldots,\eta_r$ with
\[
    \eta:=\sum_{\kappa=3}^r\eta_\kappa
    <\min\left\{\frac12,\frac{\delta_0}{8}\right\}.
\]
Multiply each of the four original weights by $1-\eta$ and assign weight
$\eta_\kappa/2$ to each member of the pair $p_\kappa,p_\kappa^\perp$. The
new moment is
\[
    (1-\eta)\frac12I_2+
    \sum_{\kappa=3}^r\frac{\eta_\kappa}{2}I_2
    =\frac12I_2.
\]
The new complex number has the form
\[
    R=(1-\eta)R_0+E,
    \qquad
    E=\sum_{\kappa=3}^r\eta_\kappa e^{4i\varphi_\kappa},
    \qquad |E|\le\eta,
\]
where $\varphi_\kappa$ is the angle of $p_\kappa$. Therefore
\[
\begin{aligned}
    \bigl||R|-t\bigr|
    &\le \bigl||R|-|R_0|\bigr|+\bigl||R_0|-t\bigr|\\
    &\le |R-R_0|+\frac{\delta_0}{4}\\
    &\le \eta|R_0|+|E|+\frac{\delta_0}{4}
    \le2\eta+\frac{\delta_0}{4}
    <\frac{\delta_0}{2}.
\end{aligned}
\]

Now suppose that $m=2r+1$ is odd. We begin with the symmetric block
\[
    u_0=(1,0),\qquad
    u_+=(\varepsilon,\zeta),\qquad
    u_-=(-\varepsilon,\zeta),
    \qquad \varepsilon^2+\zeta^2=1.
\]
The function $\varepsilon\mapsto1-4\varepsilon^2$ maps $(0,1/2)$ onto
$(0,1)$. Rational points are dense on the first-quadrant arc of
$\mathbb S^1$; for instance, the points
\[
    \left(\frac{2s}{1+s^2},\frac{1-s^2}{1+s^2}\right),
    \qquad s\in\mathbb Q\cap(0,1),
\]
are rational and dense on that arc. We may therefore choose
$\varepsilon,\zeta\in\mathbb Q$ with
\[
    0<\varepsilon<\frac12,
    \qquad
    \varepsilon^2+\zeta^2=1,
    \qquad
    |(1-4\varepsilon^2)-t|<\frac{\delta_0}{4}.
\]
Assign the weights
\[
    \mu_0=\frac{\zeta^2-\varepsilon^2}{2\zeta^2},
    \qquad
    \mu_+=\mu_-=\frac1{4\zeta^2}.
\]
Since
$\zeta^2-\varepsilon^2=1-2\varepsilon^2>0$, all three weights are positive,
and
\[
    \mu_0+\mu_++\mu_-
    =\frac{\zeta^2-\varepsilon^2+1}{2\zeta^2}=1.
\]
The off-diagonal entries in the weighted moment cancel, while
\[
    \mu_0+2\mu_+\varepsilon^2=\frac12,
    \qquad
    2\mu_+\zeta^2=\frac12.
\]
Thus this block has weighted moment $I_2/2$. Choose $\gamma\in(0,\pi/6)$ so that
$\varepsilon=\sin\gamma$ and $\zeta=\cos\gamma$. Then $u_+$ and $u_-$ have
angles $\pi/2-\gamma$ and $\pi/2+\gamma$, respectively, and the complex number
of the block is
\[
\begin{aligned}
    R_0
    &=\mu_0
      +\mu_+e^{4i(\pi/2-\gamma)}
      +\mu_-e^{4i(\pi/2+\gamma)}\\
    &=\frac{\zeta^2-\varepsilon^2+\cos4\gamma}{2\zeta^2}
      =1-4\varepsilon^2.
\end{aligned}
\]
Here the last equality follows from
$\cos4\gamma=1-8\varepsilon^2\zeta^2$ and
$\varepsilon^2+\zeta^2=1$.

If $m=3$, take only this block and set
\[
    \eta=0,
    \qquad E=0,
    \qquad R=R_0.
\]
If $m>3$, choose the remaining $r-1$ rational orthogonal pairs so that all
directions are pairwise distinct up to sign. Give these pairs positive rational total
weights $\eta_1,\ldots,\eta_{r-1}$ with
\[
    \eta:=\sum_{\sigma=1}^{r-1}\eta_\sigma
    <\min\left\{\frac12,\frac{\delta_0}{8}\right\},
\]
assign weight $\eta_\sigma/2$ to each member of the $\sigma$-th pair, and
multiply the three block weights by $1-\eta$. Since each orthogonal pair has
moment $I_2$, the new weighted moment is
\[
    (1-\eta)\frac12I_2
    +\sum_{\sigma=1}^{r-1}\frac{\eta_\sigma}{2}I_2
    =\frac12I_2.
\]
If $\psi_\sigma$ denotes the angle of the first vector in the $\sigma$-th
orthogonal pair, then that pair contributes
$\eta_\sigma e^{4i\psi_\sigma}$ to the complex number. Hence
\[
    R=(1-\eta)R_0+E,
    \qquad
    E=\sum_{\sigma=1}^{r-1}\eta_\sigma e^{4i\psi_\sigma},
    \qquad |E|\le\eta.
\]
In the case $m=3$, the same formulas hold with the sums omitted. In both
subcases, the moment is $I_2/2$, and, using $|R_0|\le1$, we obtain
\[
\begin{aligned}
    \bigl||R|-t\bigr|
    &\le |R-R_0|+\bigl||R_0|-t\bigr|\\
    &\le \eta|R_0|+|E|+\frac{\delta_0}{4}\\
    &\le2\eta+\frac{\delta_0}{4}
    <\frac{\delta_0}{2}.
\end{aligned}
\]

In both parity cases, all directions are rational and pairwise distinct up to
sign. After clearing a common denominator, Lemma~\ref{lem:general-eigenmap}
produces a linearly full flat minimal immersion on a compact finite-covering torus.
The moment identity is positive definite, so the directions, and hence the
integer frequencies, span $\R^2$; their integer span is therefore a rank-two
lattice. Proposition~\ref{prop:embedded-quotient} applies and descends the map to a
closed linearly full embedded flat minimal torus without changing the image or
any local invariant.
Moreover,
\[
    \bigl||R|-t\bigr|<\frac{\delta_0}{2}<\delta.
\]
Since $\delta_0\le t/2$ and $\delta_0\le(1-t)/2$, the same estimate gives
\[
    |R|>t-\frac{\delta_0}{2}\ge\frac{3t}{4}>0,
    \qquad
    |R|<t+\frac{\delta_0}{2}
    \le\frac{3t+1}{4}<1.
\]
Finally,
\[
    |(S+\lambda_2)-c|
    =|3-|R|-(3-t)|
    =\bigl||R|-t\bigr|<\delta.
\]
As $c\in(2,3)$ and $\delta>0$ were arbitrary, the realized constant values are
dense in $(2,3)$.
\end{proof}
\begin{proof}[\textbf{Proof of Theorem~\ref{thm:sequence}}]
Fix $m\ge3$ and put $q=2m-3$. Proposition~\ref{prop:full-odd} first
constructs linearly full flat minimal immersions on convenient finite-covering
tori. By Proposition~\ref{prop:embedded-quotient}, each of these maps factors
through its maximal frequency quotient as a flat minimal embedding
\[
    \widehat F_j:\widehat T_j^2\longrightarrow\Sph^{2m-1}=\Sph^{2+q}.
\]
The quotient map is a local isometry and the image is unchanged; hence linear
fullness and the local quantities $S,\lambda_1,\lambda_2$ and
$\rho_0^\perp$ are preserved. Lemma~\ref{lem:general-eigenmap} gives
$S\equiv2$, while Lemma~\ref{lem:flat-construction-matrix} gives
\[
    \lambda_1=1+|R_j|,
    \qquad
    \lambda_2=1-|R_j|,
    \qquad
    S+\lambda_2=3-|R_j|,
    \qquad
    \rho_0^\perp=4(1-|R_j|^2).
\]
The same proposition gives $0<|R_j|<1$ and $|R_j|\to1$, after passing to a tail
of the sequence if necessary. Hence $S+\lambda_2>2$ and
$S+\lambda_2\to2$. This proves Theorem~\ref{thm:sequence}.
\end{proof}
\begin{proof}[\textbf{Proof of Corollary~\ref{cor:all-q-nonfull}}]
For odd $q\ge3$ this follows from Theorem~\ref{thm:main}, which provides
linearly full embedded flat minimal tori. For an arbitrary $q\ge3$, start with
the embedded codimension-three flat-torus quotients
\[
    \widehat F_{N,\ell}:\widehat T_{N,\ell}^2\to\Sph^5
\]
from Theorem~\ref{thm:explicit-q3} and compose with a totally geodesic inclusion
$\Sph^5\subset\Sph^{2+q}$. The added normal directions have zero shape
operators, so Lu's fundamental matrix only acquires additional zero eigenvalues.
Thus $S$, $\lambda_1$, $\lambda_2$, and $S+\lambda_2$ are unchanged. The same
embedded flat-torus sequence therefore contradicts any positive gap constant
for closed embedded minimal surfaces, and hence also for closed immersed
minimal surfaces, in every codimension $q\ge3$.
\end{proof}

\begin{proof}[\textbf{Proof of Theorem~\ref{thm:genus-area}}]
Assume that
$
    S+\lambda_2\equiv c
$
for some constant $c$. Then
\[
    S=c-\lambda_2.
\]
By the Gauss equation for the minimal surface in the unit sphere, we have
\begin{equation}
\label{eq:K-lambda2-detailed}
    K=1-\frac S2=1-\frac{c-\lambda_2}{2}
    =\frac{\lambda_2-(c-2)}{2}.
\end{equation}
Integrating \eqref{eq:K-lambda2-detailed} over $M$ gives
\[
    \int_MK\,d\mu
    =\frac12\int_M\lambda_2\,d\mu
    -\frac{c-2}{2}\Area(M).
\]
By the Gauss--Bonnet formula, one has
\[
    \int_MK\,d\mu=2\pi\chi(M)=2\pi(2-2g)=4\pi(1-g).
\]
Substituting this into the preceding integral identity gives
\[
    c=2+\frac{8\pi(g-1)}{\Area(M)}
    +\frac{1}{\Area(M)}\int_M\lambda_2\,d\mu,
\]
which is \eqref{eq:gb-identity-intro}.
Assume now that $g\ge2$. Since $\lambda_2\ge0$ pointwise,
\eqref{eq:gb-identity-intro} immediately gives
\begin{equation}
\label{eq:c-lower-topological}
    c-2\ge \frac{8\pi(g-1)}{\Area(M)}.
\end{equation}
Set
\[
    d_g=\left\lfloor\frac{g+3}{2}\right\rfloor.
\]
We shall use the Yang--Yau inequality in the sharpened form
\begin{equation}
\label{eq:yang-yau-sharp-form}
    \lambda_1^\Delta(M)\Area(M)
    \le 8\pi d_g,
\end{equation}
for closed orientable Riemannian surfaces of genus $g$; see
\cite{YangYau1980,LiYau1982} and also the discussion in
\cite[Section~1.1]{Karpukhin2019}. We first consider the case $g>2$.
Karpukhin's strictness theorem \cite{Karpukhin2019} for Yang--Yau's bound in genus $g>2$
 improves \eqref{eq:yang-yau-sharp-form} to
\[
    \lambda_1^\Delta(M)\Area(M)<8\pi d_g.
\]
%Equivalently,
%\[
%    \frac{8\pi(g-1)}{\Area(M)}
%    >\frac{g-1}{d_g}\,\lambda_1^\Delta(M).
%\]
Combining this with \eqref{eq:c-lower-topological} yields
\[
    c>2+\frac{g-1}{\left\lfloor\frac{g+3}{2}\right\rfloor}
    \lambda_1^\Delta(M).
\]
It remains to treat the case $g=2$. Since $d_2=2$, the sharpened
Yang--Yau inequality \eqref{eq:yang-yau-sharp-form} gives
\[
    \lambda_1^\Delta(M)\Area(M)\le 16\pi.
\]
The constant $16\pi$ is sharp in genus two, as proved by Nayatani--Shoda
\cite{NayataniShoda2019}. %hence no strict inequality at the coefficient $1/2$ can be concluded from Yang--Yau's estimate alone. 
From the last inequality we
infer
\[
    \frac{8\pi}{\Area(M)}\ge \frac12\lambda_1^\Delta(M).
\]
Combining this with \eqref{eq:c-lower-topological}, and using $g=2$, we get
\[
    c\ge 2+\frac12\lambda_1^\Delta(M).
\]
\end{proof}

\section{Concluding remarks}

We conclude by recording several related remarks. We first recall a pinching
theorem and explain why the examples constructed above do not conflict with it.

\begin{GLZtheorem}[\cite{GLZ}]
Let $M^n$ be a closed minimal submanifold immersed in the unit sphere
$\Sph^{n+p}$. If
\[
    \lambda_1\le n,
    \qquad
    \rho^\perp\le
    \frac{1}{\sqrt2\,n(n-1)}\inf_{x\in M}(n-\lambda_1)(x),
\]
then $M^n$ must be one of the following:
\begin{enumerate}[label=\textup{(\roman*)},leftmargin=2em]
\item the great sphere with $S\equiv 0$;
\item the product
\[
    \prod_{i=1}^{r+1}\Sph^{n_i}
    \left(\sqrt{\frac{n_i}{n}}\right),
    \qquad
    \sum_{i=1}^{r+1}n_i=n,
    \qquad
    1\le r\le p.
\]
In the second case, $\rho^\perp=0$, $\lambda_1=n$, and $S=rn$.
\end{enumerate}
\end{GLZtheorem}

\begin{remark}
The flat tori constructed here do not satisfy the Ge--Li--Zhang pinching
hypothesis. For the explicit examples, put $\tau=\ell/N$. Then
\[
    n=2,
    \qquad
    2-\lambda_1=\tau^2,
    \qquad
    \rho^\perp=\tau\sqrt{2-\tau^2}.
\]
The pinching condition would require
\[
    \tau\sqrt{2-\tau^2}\le \frac{\tau^2}{2\sqrt2},
\]
which is impossible for $0<\tau<1$. Thus there is no conflict: although
$\lambda_1\le2$, the normal scalar curvature is too large relative to
$\frac{1}{2\sqrt2}(2-\lambda_1)$.
\end{remark}

\begin{remark}%[Relation with the Chern problem]
The examples do not affect the Chern problem for constant values of $S$. Along
the whole family one has $S\equiv2$. The varying quantity is
\[
    S+\lambda_2=2+\frac{\ell^2}{N^2}\to2.
\]
Thus the construction is a counterexample to a second gap for Lu's refined
quantity, not to discreteness of the possible values of $S$.
\end{remark}

\begin{remark}%[Self-shrinker rescalings]
Every minimal surface $F:M^2\to\Sph^{N-1}(1)$ gives a compact self-shrinker in
$\R^N$ after the standard rescaling $X=2F$. Indeed, Takahashi's equation implies
\[
    \mathbf H_X+\frac{X^\perp}{2}=0,
\]
where $\mathbf H_X$ denotes the trace mean-curvature vector of the Euclidean
immersion $X$. Thus the flat minimal tori constructed above also admit associated
compact flat self-shrinker rescalings. For these rescalings, $|A_X^\R|^2\equiv1$,
since the spherical examples have $S\equiv2$. We record this only as a simple
consequence of the spherical construction and do not use it in the proof of the
Lu-gap results.
\end{remark}

\begin{remark}%[Codimension two and even codimension]
The present method gives linearly full examples only in odd codimensions. The
even-codimensional examples mentioned above are obtained by placing the
codimension-three family in a larger totally geodesic sphere, and so they
are not linearly full in the larger ambient sphere. In particular, the
linearly full codimension-two case is not addressed by the present construction.
It has subsequently been settled in the positive direction by Ge--Li--Zhang
\cite{GeLiZhang26}, who proved that Lu's second-gap conjecture holds for all
closed minimal surfaces in $\Sph^4$.
\end{remark}
\begin{remark}%[Arithmetic nature of the construction]
In our construction, closedness is encoded by the frequency lattice: the
trigonometric coordinate functions descend to a compact flat torus precisely
because the frequency vectors lie in the corresponding dual lattice.  In the
explicit families constructed above, this is achieved through integer or
rational frequency data.  Thus the closed examples obtained here form an
arithmetic, hence countable, family, although the corresponding constant values
of \(S+\lambda_2\) are dense in intervals.  This should be distinguished from
the existence of a continuous moduli family of closed embedded minimal
submanifolds.
\end{remark}

\begin{question}
Apart from the flat-torus degenerations constructed in this paper, for which
classes of closed minimal submanifolds with constant $S+\lambda_2$ does Lu's
second-gap conclusion remain valid? In the surface case, natural test cases
include non-flat minimal surfaces in codimension $q\ge3$, linearly full examples
in even codimensions $q\ge4$, families satisfying normal-curvature pinching
conditions, and classes with an \emph{a priori} area bound. More generally, one may ask for geometric or
topological hypotheses that exclude the flat-torus escape described above and
restore a positive gap for $S+\lambda_2$ above the first-gap value.
\end{question}

\begin{acknow}
The authors are grateful to Professors Zhiqin Lu, Zizhou Tang, Jianquan Ge, and
Wenjiao Yan for their encouragement, support, and valuable discussions.
\end{acknow}

\end{document}